\documentclass[11pt, oneside]{article}   	
\usepackage{geometry}                		
\usepackage{graphicx}				
\usepackage{amssymb,amsthm,xcolor}					
\usepackage[square,sort,comma,numbers]{natbib}

\usepackage{amsmath}
\usepackage{epstopdf}
\usepackage{bbm}
\usepackage{mathtools}

\newcommand{\dotp}[2]{\left\langle #1, #2\right\rangle}

\newcommand{\bfx}{{\bf x}}

\newcommand{\bfxi} {\mbox{\boldmath $\xi$}}

\newcommand{\abs}[1]{\left\vert#1\right\vert}

\newcommand{\norm}[1]{\big\Vert#1\big\Vert}

\newcommand\wt[1]{{ \widetilde{#1} }}

\newcommand \bbP{\mathbb{P}}
\newcommand \bbE{\mathbb{E}}

\newcommand \bft{\boldsymbol \theta}
\def\T{{ \mathrm{\scriptscriptstyle T} }}
\def\TV{{ \mathrm{\scriptscriptstyle TV} }}
\def\Gauss{{ \mathrm{N} }}
\def\d{{ \mbox{d} }}
\newcommand{\ind}{\mathbbm{1}}

\def\m{\mathcal}
\def\mb{\mathbb}
\def\mr{\mathrm}

\def\tr{{\rm tr\,}}

\newtheorem{theorem}{Theorem}[section]
\newtheorem{lemma}[theorem]{Lemma}

\newtheorem{corollary}[theorem]{Corollary}
\newtheorem{rem}{Remark}[section]

\begin{document}

\title{Posterior contraction in Gaussian process regression using Wasserstein approximations}

\author{ {\bf Anirban Bhattacharya}\\
Department of Statistics, Texas A \&M University, College Station, TX, \\email: anirbanb@stat.tamu.edu
\\{\bf Debdeep Pati}, \\Department of Statistics, Florida State University, Tallahassee, FL, \\email: debdeep@stat.fsu.edu}

\maketitle

\begin{center}
\textbf{Abstract}
\end{center}
We study posterior rates of contraction in Gaussian process regression with unbounded covariate domain. Our argument relies on developing a Gaussian approximation to the posterior of the leading coefficients of a Karhunen--Lo{\`e}ve expansion of the Gaussian process. The salient feature of our result is deriving such an approximation in the $L^2$ Wasserstein distance and relating the speed of the approximation to the posterior contraction rate using a coupling argument. Specific illustrations are provided for the Gaussian or squared-exponential covariance kernel. 
\vspace*{.3in}

\section{Introduction}

Gaussian process (GP) priors \cite{rasmussen2006gaussian} are popularly used in a variety of machine learning applications including regression, classification, density estimation, latent variable modeling, unsupervised learning to name a few. GP priors also share a deep connection with frequentist reproducible kernel Hilbert space (RKHS) based regularization methods; see, for example Chapter 6 of \cite{rasmussen2006gaussian}. Paralleling the development of scalable algorithms for GP regression, there has been substantial progress in recent years in understanding frequentist properties of the posterior arising from a Gaussian process prior. A standard way of evaluating frequentist properties of Bayesian procedures is to consider whether the amount of posterior mass assigned to a neighborhood of the true data-generating parameter (a function in the present setting) converges to one with increasing sample size. If the neighborhood size is fixed, the above phenomenon is termed {\em posterior consistency}, while if the neighborhood size is allowed to shrink to zero, then the (best possible) shrinking rate is termed the {\em posterior contraction rate}. \cite{ghosal2006posterior,choi2007posterior} established posterior consistency of GP priors, while posterior contraction rates in a variety of contexts were derived in \cite{van2007bayesian,van2008rates,van2009adaptive,van2011information,aniso2014,pati2015rd} among others; see also \cite{seeger2008information} for an information-theoretic approach. In particular, it has been established in  various contexts that the posterior distribution contracts at an optimal rate (up to a logarithmic term) in a frequentist minimax sense.

The above references exclusively deal with compactly supported functions as parameters, even though the priors in principle are random functions on full Euclidean spaces. In fact, the influential article \cite{van2009adaptive} remarks that 
\begin{center}
{\em `` Consistency of a posterior on the full space can be expected only if the tails of the functions are restricted. If they are not, then one would still expect that the posterior restricted to compact subsets contracts at some rate. At the moment there seem to exist no results that would yield such a rate (or even consistency)''}.
\end{center}
In this article, we take a step towards addressing this question borrowing inspiration from the kernel regression literature \cite{geer2000empirical,steinwart2009optimal}, where a $L^2$ norm weighted by a possibly unbounded covariate density is commonly used as a measure of discrepancy. We focus on the nonparametric regression model with Gaussian errors
\begin{align}\label{eq:npreg}
Y_i = f(X_i) + \epsilon_i, \quad \epsilon_i \sim \Gauss(0, \sigma^2), \quad i = 1, \ldots, n,
\end{align}
where $X_i \in \m X$ are covariates and $f : \m X \to \mb R$ is an unknown regression function with possibly unbounded domain $\mathcal{X} \subset \mb R^d$, which is assigned a zero-mean GP prior. We operate in a random design setting where the covariates are drawn according to a distribution $\rho$ on $\m X$ and study contraction of the posterior in an $L^2(\rho)$ norm, i.e., the $L^2$ norm on $\mb R^d$ weighted with respect to the covariate density $\rho$. This choice ensures that the large covariate values are weighted down, which can be considered as a way of restricting the tails of the function as in the comment by \cite{van2009adaptive} above. 


In deriving the posterior rate of contraction, we expand the GP prior via a Karhunen--Lo{\`e}ve expansion \cite{adler1990introduction} and then derive a Gaussian approximation to the posterior distribution of the leading coefficients of the expansion. The Gaussian approximation is derived in an $L^2$ Wasserstein metric which is particularly suited for the present situation for reasons described in the sequel. Using a careful coupling argument, the speed of such Gaussian approximations are related to the posterior contraction rate; a result which is new to best of our knowledge. Another key ingredient of our method is to control the effect of the truncating the Karhunen--Lo{\`e}ve expansion in the posterior. This typically requires bounds on the concentration of the prior around the true function in the sup-norm, which is difficult to control for unbounded covariates. A second contribution of this paper is to develop a general result (Theorem \ref{thm:denom}) 
to bound (with high probability) the integrated log-likelihood ratio from below by a quantity involving prior concentration around the true function in the $L^2_{\rho}$ norm instead of the sup-norm. We believe this result may be of independent interest in random design Gaussian regression. We may comment here that in addition to dealing with unbounded covariates, the proposed technique has an added advantage of making the bias-variance tradeoff in the posterior explicit as in kernel ridge regression theory.

While we make general assumptions on the covariance kernel to prove our results, verifying them in a specific context requires suitable control over the eigenfunctions of the kernel. This can potentially be a non-trivial exercise, in particular if the covariance kernel involves a parameter which is sample-size dependent. We illustrate this in case of a squared-exponential kernel, for which explicit expressions of the eigenfunctions are available \cite{rasmussen2006gaussian}. We develop precise bounds on the eigenfunctions making the role of a scale parameter explicit, which should be more broadly useful. 

\section{Preliminaries}\label{sec:prelim}

For a square matrix $B$, $\tr(B)$ and $\abs{B}$ respectively denote the trace and the determinant of $B$. If $B$ is positive semi-definite (psd), then let $B^{1/2}$ denote its unique psd square-root, so that $(B^{1/2})^2 = B$. $B$ is positive definite (pd) if and only if $B^{1/2}$ is pd \cite{bhatia2009positive}, and in such cases we can unambiguously define $B^{-1/2} = (B^{-1})^{1/2}$. Given two pd matrices $B_1$ and $B_2$, we write $B_1 \succsim B_2$ if $B_1 - B_2$ is psd. 
For a $p \times d$ matrix $A = (a_{jj'})$ with $p \ge d$, the singular values of $A$ are the eigenvalues of $(A^{\T}A)^{1/2}$. We shall use $s_{\max}(A)$ and $s_{\min}(A)$ to denote the largest and smallest non-zero singular values respectively; the condition number $\kappa(A) = s_{\max}(A)/s_{\min}(A)$. The Frobenius norm ($\norm{\cdot}_F$)  and  the operator norm  ($\norm{\cdot}_2$) are defined in the usual way, with $ \norm{A}_F := \sqrt{\tr(A^{\T} A)}$ and $\norm{A}_2 := s_{\max}(A)$. Note that $\norm{A}_2^2 = \norm{A^{\T} A}_2$. 

For a vector $ x \in \mathbb{R}^d$,  $\norm{x}$ will denote its Euclidean norm.  Let $\ell_2 = \{\bft = (\theta_1, \theta_2, \ldots):  \sum_{j=1}^{\infty} \theta_j^2 < \infty\}$ denote the space of square-summable sequences, with $\norm{\bft}_{\ell_2} = (\sum_{j=1}^{\infty} \theta_j^2)^{1/2}$. Let $\Theta_{\alpha} = \{\bft \in \ell_2:  \sum_{j=1}j^{2\alpha}\theta_j^2  < \infty\}$ denote the Sobolev space of sequences with ``smoothness'' $\alpha > 0$, and denote the Sobolev norm $\norm{\theta}_{\alpha} =   (\sum_{j=1}j^{2\alpha}\theta_j^2)^{1/2}$. For a density $\rho$ on $\m X \subset \mb R^d$, let $L^2_{\rho}(\m X) = \{g: \int g(x)^2 \rho(x) dx < \infty\}$ denote the space of square-integrable functions with respect to $\rho$. $L^2_{\rho}(\m X)$ is a Hilbert space under the inner product $\dotp{g_1}{g_2} = \int g_1(x) g_2(x) \rho(x) dx$; the resulting norm will be denoted by $\norm{\cdot}_{2, \rho}$, so that $\norm{g}_{2, \rho}^2 = \int g(x)^2 \rho(x) dx$. 

Throughout $C, C', C_1, C_2, \ldots$ are generically used to denote positive constants whose values might change from one line to another, but are independent from everything else.  $\lesssim / \gtrsim$ denote inequalities upto a constant multiple.  $a \asymp b$  when we have both $a \lesssim b$ and $a \gtrsim b$. 

\subsection{{\bf The $L^p$ Wasserstein distances}}
Given two probability measures $P$ and $Q$ on $\mb R^d$, the total variation distance $d_{\TV}(P, Q)
:= \sup_{A} |P(A) - Q(A)|$ where the supremum is over all Borel subsets of $\mb R^d$ and the Kullback--Leibler divergence $D(P || Q)$ are defined in the usual way.
For $p \ge 1$, the $L^p$ Wasserstein distance with respect to the Euclidean metric (henceforth $W_p$ in short), denoted $d_{W, p}(P, Q)$, is defined as
\begin{align}\label{eq:wstein_def}
d_{W, p}(P, Q) = \inf_{\mbox{joint}(P, Q)} (\bbE \norm{X - Y}^p)^{1/p},
\end{align}
where $\mbox{joint}(P, Q)$ denotes all random vectors $(X, Y) \in \mb R^d \times \mb R^d$, such that $X \sim P, Y \sim Q$. The Wasserstein distances have their origins in the problem of {\em optimal transport}; refer to \cite{givens1984class,gelbrich1990formula} for background and properties. Explicit expressions are available for the $W_2$ distance between two $d$-dimensional Gaussian measures. In particular, if $P \equiv \Gauss_d(\mu_1, \Sigma_1), Q \equiv \Gauss_d(\mu_2, \Sigma_2)$ and $\Sigma_1 \Sigma_2 = \Sigma_2 \Sigma_1$, then
\begin{align}\label{eq:wstein_normal} 
d_{W,2}^2(P, Q) = \norm{\mu_1 - \mu_2}^2 + \norm{\Sigma_1^{1/2} - \Sigma_2^{1/2}}_F^2.
\end{align}
For $d = 1$, the $W_2$ distance is identical to the Fr{\'e}chet distance \cite{frechet1957distance}. 

\section{Posterior contraction in random design GP regression}\label{sec:rdgp}

Write the nonparametric regression model \eqref{eq:npreg} in vector form as 
\begin{align}\label{eq:normal_mean}
Y = F + \varepsilon,  \quad \varepsilon \sim \Gauss(0, \sigma^2 \mr I_n),
\end{align}
where $Y = (Y_1, \ldots, Y_n)^{\T}$ and $F = (f(X_1), \ldots, f(X_n))^{\T}$. We shall assume the error variance $\sigma^2$ to be known throughout this paper. Let $f_0 : \m X \to \mb R$ denote the true data generating function and define $F_0 = (f_0(X_1), \ldots, f_0(X_n))^{\T}$. 

As mentioned in the Introduction, we operate in a random design setting where we assume that the covariates $X_i$ are independent and identically distributed according to a known density $\rho$ on $\m X$ and $Y_i \mid X_i \sim \Gauss(f_0(X_i), \sigma^2)$ independently for $i = 1, \ldots, n$. Letting $h(y, x) = \Gauss(y \mid f_0(x), \sigma^2) \rho(x)$, the true joint density of $(Y, X)$ is an $n$-fold product of $h$. We shall use $\bbE_0$ to denote an expectation with respect to the true joint distribution of $(Y, X)$; $\bbE_X$ and $\bbE_{0\mid X}$ will respectively denote an expectation with respect to the marginal distribution of $X$ and the conditional of $Y$ given $X$. Similarly, $\bbP_0, \bbP_X$ and $\bbP_{0 \mid X}$ will denote probabilities under the respective distributions. 

Consider a $\mbox{GP}(0, \sigma^2 K)$ prior on $f$, where $K(\cdot, \cdot)$ is a positive definite correlation function, i.e., $K(x,x) = 1$ for all $x \in \m X$. We shall generically use $\Pi$ and $\Pi(\cdot \mid Y, X)$ to denote the prior and posterior distribution of $f$. 
Under suitable regularity conditions, Mercer's theorem \cite{adler1990introduction} guarantees that the kernel $K$ admits an eigen-expansion of the form $K(x, x') = \sum_{j=1}^{\infty} \lambda_j \phi_j(x) \phi_j(x')$ in $L^2_{\rho}(\m X)$, where $\{\phi_j\}$ is an orthonormal system in $L^2_{\rho}(\m X)$ ($\int \phi_j(x) \phi_l(x) \rho(x) dx = \delta_{jl}$) and $\{\lambda_j\}$ the corresponding non-negative eigenvalues, which satisfy  
\begin{align}\label{eq:eigenfn}
\int K(x, x') \phi_j(x') \rho(x') dx' = \lambda_j \phi_j(x), \quad j = 1, 2, \ldots
\end{align}
As a concrete example, consider the squared-exponential kernel $K_a(x, x') = \exp(-a^2 \norm{x - x'}^2)$ indexed by a length-scale parameter $a$. For Gaussian covariate distributions $\rho$, explicit expressions for the eigenfunctions and eigenvalues are known \cite{rasmussen2006gaussian}. Specifically, when the dimension $d = 1$, with a Gaussian covariate density 
$\rho(x) = \sqrt{2b/\pi} \, e^{-2 b x^2}$ and $c = \sqrt{b^2 + 2 b a^2}$,  
{\small\begin{align}\label{eq:hermitebasis}
\phi_j(x) = \frac{(c/b)^{1/4}}{ \sqrt{ 2^{j-1} \, (j - 1) ! } } \, e^{-(c - b) x^2}  H_{j-1}(\sqrt{2c} \, x), \quad \lambda_j = \bigg( \frac{2b}{b + a^2 + c} \bigg)^{1/2} \, \bigg( \frac{a^2}{ b + a^2 + c} \bigg)^{j-1},
\end{align}}
where $H_k(x) = (-1)^n e^{x^2} \frac{d^k}{dx^k} e^{-x^2}, k = 0, 1, \ldots$ denote the Hermite polynomials\footnote{Many references term $H_k$s the ``physicist's Hermite polynomial'' to distinguish from the ``probabilist's Hermite polynomial'' $h_k(x) = 2^{-k/2}H_k(x/\sqrt{2})$}. We shall return to the squared-exponential kernel in Section 4. 

By the Karhunen--Lo{\`e}ve Theorem \cite{adler1990introduction}, the GP itself can be expanded as 
\begin{align}\label{eq:kar_lo}
f(x) = \sigma \sum_{j=1}^{\infty} \sqrt{\lambda_j} \ Z_j \phi_j(x),
\end{align}
where $Z_j$s are i.i.d. $\Gauss(0, 1)$. If the series representation above is truncated to the first $k$ terms and the resulting random function is denoted by $f^t$, then it follows from \eqref{eq:eigenfn} and the orthogonality of the eigenfunctions $\phi_j$ that $E \norm{f - f^t}_{2, \rho}^2 = \sigma^2 \sum_{j=k+1}^{\infty} \lambda_j$. The accuracy of the truncation relies on the rate of decay of the eigenvalues, which is related to the smoothness of the GP. For example, if the sample paths of a GP are infinite smooth, then the eigenvalues decay exponentially fast, so that relatively few leading terms in the expansion \eqref{eq:kar_lo} offer a close reconstruction of the original process. 

Given a $\mbox{GP}(0, \sigma^2 K)$ prior, we shall consider such truncations of \eqref{eq:kar_lo} to define priors which we refer to as truncated Gaussian process (tGP) priors: 
\begin{align}\label{eq:tke}
f^t(x) = \sum_{j=1}^{k_n} \theta_j \phi_j(x), \quad \theta_j \sim \Gauss(0, \sigma^2 \lambda_j).
\end{align}
Let $\theta^t = (\theta_1, \ldots, \theta_{k_n})^{\T}$ denote the $k_n$-dimensional vector of coefficients in \eqref{eq:tke} and $\Lambda = \mbox{diag}(\lambda_1, \ldots, \lambda_{k_n})$, so that $\theta^t \sim \Gauss(0, \sigma^2 \Lambda)$.
One may consider the tGP priors \eqref{eq:tke} as {\em sieve} approximations to the original GP prior, where the basis functions $\phi_j$s and the prior variances $\lambda_j$s are determined by the choice of the kernel $K$. We denote such priors by $\mbox{tGP}_{k_n}(0, K)$; the truncation level $k_n$ will be suppressed when clear from the context. 

We note here that the tGP prior is solely introduced to obtain theoretical understanding of the original GP prior and the resulting posterior. When working with a tGP prior, one can conveniently direct attention to the coefficient vector $\theta^t$, which is finite-dimensional; albeit with the dimension possibly increasing with sample size $n$. In fact, defining the $n \times k_n$ (random) matrix $\Phi = (\phi_j(X_i))_{1 \le i \le n, 1 \le j \le k_n}$, one can write model \eqref{eq:normal_mean} equipped with a tGP prior \eqref{eq:tke} as 
\begin{align}\label{eq:tgp_vec}
Y \sim \Gauss(F, \sigma^2 \mr I_n), \quad F = \Phi \theta^t, \quad \theta^t \sim \Gauss_{k_n}(0, \sigma^2 \Lambda).
\end{align}
Using standard Gaussian conjugacy, the posterior distribution of $\theta^t$ under \eqref{eq:tgp_vec} is 
\begin{align}\label{eq:tgp_posterior}
\wt{\m W}^t(\cdot \mid X, Y) \equiv \Gauss\big(\wt{\theta}, \wt{\Sigma}\big), \quad \wt{\theta} = (\Phi^{\T} \Phi + \Lambda^{-1})^{-1} \Phi^{\T} Y, \quad \wt{\Sigma} = \sigma^2 (\Phi^{\T} \Phi + \Lambda^{-1})^{-1}.
\end{align}
From \eqref{eq:tgp_posterior}, the posterior distribution of $F = \Phi \theta^t$ is also Gaussian. 
With a slight abuse of terminology, we shall refer to the posterior distribution \eqref{eq:tgp_posterior} of $\theta^t$ as the {\em tGP posterior} induced by the tGP prior. The role of the tGP in deriving posterior rates of contraction for the original GP prior is made precise through the following general rate theorem for GP priors. 
We first state our assumption regarding the true data generating function $f_0$ and introduce some notations. 
\begin{description}
\item[(T1)] The true data generating function $f_0 \in L^2_{\rho}(\m X)$, so that $f_0 = \sum_{j=1}^{\infty} \theta_{0j} \phi_j$ with $\theta_{0j} = \dotp{f_0}{\phi_j} := \int f_0(x) \phi_j(x) \rho(x) dx$. The convergence of the infinite sum is in an $L^2_{\rho}$ sense, i.e., $\norm{f_0 - \sum_{j=1}^J \theta_{0j} \phi_j}_{2, \rho} \to 0$ as $J \to \infty$.

Define $f_0^{t} = \sum_{j=1}^{k_n} \theta_{0j} \phi_j$, $\theta_0^{t} = (\theta_{0j})_{1 \le j \le k_n} \in \mb R^{k_n}$. Also define $\norm{\theta_0^t}_{\mb H}^2 = \sum_{j=1}^{k_n} \frac{\theta_{0j}^2}{\lambda_j} = \norm{\Lambda^{-1/2} \theta_0^t}^2$.
\end{description}

\begin{theorem}\label{thm:rd_gp_gen}
Consider model \eqref{eq:npreg} with a GP prior $f \sim \mbox{GP}(0, \sigma^2 K)$, where the kernel $K$ has eigenfunctions $\{\phi_j\}$ and eigenvalues $\{\lambda_j\}$ with respect to the covariate density $\rho$ as in \eqref{eq:eigenfn}. Assume the true function $f_0$ satisfies $\mathrm{ {\bf (T1) }}$. For $k_n < n$, let $\wt{\m W}^t(\cdot \mid X, Y)$ denote the tGP posterior as in \eqref{eq:tgp_posterior}. 
Let $\epsilon_n \to 0$ be a sequence with $n \epsilon_n^2 \to \infty$ and $\norm{f_0 - f_0^t}_{2, \rho} \lesssim \epsilon_n$.  
Then, for any $M > 0$, 
\begin{align}
\bbE_0 \Pi(\norm{f - f_0}_{2, \rho} > M \epsilon_n \mid Y, X) \le T_{1n} + T_{2n}, 
\label{eq:rates_gp_gen} 
\end{align}
where
{\small \begin{align}
&T_{1n} = \frac{\bbE_0\left\{ \ind_{A_n}(X) \d_{W,2}^2 \bigg[\wt{\m W}^t(\cdot \mid Y, X), \Gauss_{k_n}\bigg(\theta_0^t, \frac{\sigma^2}{n} \mr I_{k_n}\bigg)  \bigg] \right\} }{M^2 \epsilon_n^2/4} + P(\chi_{k_n}^2 > M^2 n \epsilon_n^2/4)  +  \bbP_{X}(A_n^c), \label{eq:t1n_gen} \\
& T_{2n} = \bbE_0 \Pi(\norm{f - f^t}_{2, \rho} > M \epsilon_n \mid Y, X). \label{eq:t2n_gen}
\end{align}}
In \eqref{eq:t1n_gen}, $\chi_{r}^2$ denotes a $\chi^2$ random variable with $r$ degrees of freedom and $A_n \subset \m X^n$ is any set in the $\sigma$-field generated by $X_1, \ldots, X_n$. 
\end{theorem}
It immediately follows from \eqref{eq:rates_gp_gen} that for a given $\epsilon_n$, if the sequences $T_{1n}, T_{2n} \to 0$, then $\epsilon_n$ is an upper bound to the posterior contraction rate \cite{ghosal2000convergence} in the $L^2_{\rho}$ norm; note that no assumptions regarding the support of the covariate density $\rho$ is made. Theorem \ref{thm:rd_gp_gen} thus relates the posterior contraction rate of a GP prior to (i) the speed of a posterior Wasserstein approximation of the induced tGP prior ($T_{1n}$), and (ii) the associated truncation error ($T_{2n}$). To obtain the best possible rate out of Theorem \ref{thm:rd_gp_gen}, one needs to choose the truncation level $k_n$ (and to a lesser extent the set $A_n$) in an optimal fashion. 
The role of these quantities will become more explicit once we provide manageable  bounds to $T_{1n}$ and $T_{2n}$ in the subsequent sections. To that end, we need to make additional assumptions on the eigenfunctions $\{\phi_j\}$ and eigenvalues $\{\lambda_j\}$ of the kernel $K$ stated below. 
Recall $\Lambda = \mbox{diag}(\lambda_1, \ldots, \lambda_{k_n})$ and $\Phi = (\phi_j(X_i))_{1 \le i \le n, 1 \le j \le k_n}$. Assume
\begin{description}
\item [(A1)] $\norm{\Lambda^{-1}}_2 < n/4$. 
\item [(A2)] $\sup_{x \in \m X} |\phi_j(x)| \le L_n$ for all $j = 1, \ldots, k_n$, with $L_n^2 k_n \log k_n < n$.
\end{description}
Assumption {\bf (A1)} typically implies a bound on the growth rate of $k_n$; for example, if the eigenvalues decay polynomially, $\lambda_j \asymp j^{-2 \beta}$ for some $\beta > 0$, then $\norm{\Lambda^{-1}}_2 = \lambda_{k_n}^{-1} \asymp k_n^{2 \beta}$ and hence {\bf (A1)} is satisfied for all $k_n \precsim n^{1/(2 \beta + 1)}$. Assumption {\bf (A2)} is readily satisfied if all the eigenfunctions $\phi_j$ are uniformly bounded in magnitude by a constant. However, {\bf (A2)} is more general and allows the sup-norm of the top $k_n$ eigenfunctions to increase with $n$ subject to a growth condition; note that no assumption is made regarding the trailing eigenfunctions. Allowing the sup-norm bound to grow with $n$ is important when the kernel is indexed by one or more hyper parameters which may depend on $n$. A specific illustration is provided in the context of the squared-exponential covariance kernel \eqref{eq:hermitebasis} in Section 4. It turns out a non-trivial exercise to bound the eigenfunctions \eqref{eq:hermitebasis} making the dependence on the bandwidth parameter $a$ explicit.

\subsection{{\bf Wasserstein approximations to tGP posteriors}}

To bound $T_{1n}$, one primarily needs a handle on the squared $W_2$ distance between the tGP posterior $\wt{\m W}^t(\cdot \mid Y, X)$ in \eqref{eq:tgp_posterior} and a Gaussian $\Gauss_{k_n}\big(\theta_0^t, \sigma^2  \mr I_{k_n}/n \big)$ distribution. Inspecting the proof of Theorem \ref{thm:rd_gp_gen}, it may seem a more obvious choice for $T_{1n}$ is 
\begin{align*}
T_{1n}^* = \bbE_0 d_{\TV} \bigg[\wt{\m W}^t(\cdot \mid Y, X), \Gauss_{k_n}\bigg(\theta_0^t, \frac{\sigma^2}{n} \mr I_{k_n}\bigg)  \bigg] + P(\chi_{k_n}^2 > M^2 n \epsilon_n^2/4),
\end{align*}
where $d_{\TV}$ denotes the total variation distance. However, such approximations in the total variation distance require a {\em prior flatness} condition \cite{bontemps2011bernstein} which is not satisfied by the tGP priors. We find that bounding the $W_2$ distance between the tGP posterior and the asymptotic Gaussian distribution is less demanding in the present setting compared to the total variation distance. However, the connection between such an approximation result in the $W_2$ distance and posterior contraction rates in the $L^2_{\rho}$ norm is not immediately clear. We devise a coupling argument to relate the two quantities in the proof of Theorem \ref{thm:rd_gp_gen}.

The $\ind_{A_n}(X)$ term in $T_{1n}$ is introduced as a technical device to control the expectation of the squared Wasserstein distance on $A_n$;  we appropriately choose $A_n$ in a way so that $A_n^c$ receives vanishingly small probability under $\bbP_X$. Before proceeding further, we settle with a choice of $A_n$ in the following Lemma \ref{lem:well_cond}.
\begin{lemma}\label{lem:well_cond}
Assume the eigenfunctions $\{\phi_j\}$ of the kernel $K$ with respect to the covariate density $\rho$ satisfy {\bf (A2)}. Define $A_n = \{ \norm{ \Phi^{\T} \Phi - n \mr I_{k_n} }_2 < n/2\}$. Then, $\bbP_X(A_n^c) < k_n e^{-C n/(k_n L_n^2)}$. 
\end{lemma}
\begin{rem}\label{rem:well_cond}
On the set $A_n$, $\Phi$ satisfies 
\begin{align}\label{eq:well_cond1}
\norm{\Phi^{\T} \Phi}_2 \le 3n/2, \quad s_{\min}(\Phi^{\T} \Phi) \ge n/2, \quad \kappa(\Phi^{\T} \Phi) \le 3, \quad \tr[(\Phi^{\T} \Phi)^{-1}] \le \frac{2 k_n}{n}. 
\end{align}
\end{rem}
Lemma \ref{lem:well_cond} follows from a measure concentration phenomenon which under appropriate conditions on the summands ensures that a sum of independent symmetric random matrices is concentrated around its expectation with high probability. We can write $\Phi^{\T} \Phi = \sum_{i=1}^n \phi^{(i)} (\phi^{(i)})^{\T}$ with $\phi^{(i)} = (\phi_j(X_i))_{1\le j \le k_n} \in \mb R^{k_n}$ independent for $i = 1, \ldots, n$. Using the orthonormality of the eigenfunctions $\{\phi_j\}$, $(\bbE_X \phi^{(i)} (\phi^{(i)})^{\T})_{jl} = \int \phi_j(x) \phi_l(x) \rho(x) dx = \delta_{jl}$ and hence $\bbE_X \Phi^{\T} \Phi = n \mr I_{k_n}$. We specifically apply a version of matrix Bernstein inequality \cite{tropp2012user} to prove the concentration of $\Phi^{\T} \Phi$ around $n \mr I_{k_n}$; the proof is deferred to Section \ref{sec:pf_main}. The sup-norm bound on the eigenfunctions $\phi_j$s in Lemma \ref{lem:well_cond} is used to bound the operator norms of the matrices $\phi^{(i)} (\phi^{(i)})^{\T}$.

We are now in a position to state our approximation result in the $W_2$ distance that provides a simple bound to the first term of $T_{1n}$ in \eqref{eq:t1n_gen}. Recall $\theta_0^t, f_0^t$ from {\bf (T1)}. 
\begin{theorem}\label{thm:rd}
Assume the true function $f_0$ satisfies {\bf (T1)} and the eigenfunctions $\{\phi_j\}$ and eigenvalues $\{\lambda_j\}$ of the kernel $K$ with respect to the covariate density $\rho$ satisfy {\bf (A1)} and {\bf (A2)}. Let $A_n \in \m X^n$ be the set defined in Lemma \ref{lem:well_cond}. The tGP posterior $\wt{\m W}^t(\cdot \mid Y, X)$ from \eqref{eq:tgp_posterior} satisfies
\begin{align}
\bbE_{0}  \left\{ \ind_{A_n}(X) \d_{W,2}^2 \bigg[\wt{\m W}^t(\cdot \mid Y, X), \Gauss_{k_n}\bigg(\theta_0^t, \frac{\sigma^2}{n} \mr I_{k_n}\bigg)  \bigg] \right\} \lesssim \sigma^2 \frac{k_n}{n} + \frac{\norm{\theta_0^t}_{\mb H}^2}{n} + \norm{f_0 - f_0^t}_{2, \rho}^2. \label{eq:bvm_w2_rd1}
\end{align}
\end{theorem}
While $f_0$ is only assumed to be an element of $L^2_{\rho}(\m X)$ in Theorem \ref{thm:rd}, additional smoothness assumption can be utilized to obtain more precise bounds on the truncation error $\norm{f_0 - f_0^t}_{2, \rho}^2$ in \eqref{eq:bvm_w2_rd1}. The bound \eqref{eq:bvm_w2_rd1} indicates a typical bias-variance type tradeoff: increasing the truncation level $k_n$ will improve the truncation error $\norm{f_0 - f_0^t}_{2, \rho}^2$, however at the expense of the first two terms increasing. Typically, if $f_0$ is $\alpha$-smooth, then the first two summands contribute a $k_n/n$ factor and the truncation error is of the order $k_n^{-2 \alpha}$; with $k_n/n + k_n^{-2 \alpha}$ attaining its minimum when $k_n = n^{1/(2 \alpha + 1)}$. The $\norm{\theta_0^t}_{\mb H}^2$ term can be considered an RKHS type penalty; indeed, it is the RKHS norm of $\theta_0^t$ relative to a $\Gauss(0, \Lambda)$ distribution \cite{van2008reproducing}. 

\subsection{{\bf Handling the truncation error}}

We now focus attention on the term $T_{2n}$ in \eqref{eq:rates_gp_gen}. To this end, we rely on a standard argument in Bayesian nonparametrics: if the prior probability of a set is exponentially small, then its posterior probability converges to zero. Such an argument is commonly used to derive upper \cite{ghosal2000convergence} and lower \cite{castillo2008lower} bounds to the posterior convergence rate. However, a crucial ingredient for the above argument to work is to obtain suitable lower bounds to the log-likelihood ratio integrated with respect to the prior. The only such result that we are aware of in the random design setting is from \cite{van2011information}, who derive a bound for the empirical $L_2$ norm and then use a functional Bernstein inequality to extrapolate to the $L^2_{\rho}$ norm.  Their result requires the prior draws from the GP to be bounded with probability one, which may not be the case for non-compact covariates. In Theorem \ref{thm:denom} below, we develop a general result
to bound (with high probability) the integrated log-likelihood ratio from below by a quantity involving the prior concentration around the true function in the $L^2_{\rho}$ norm.  A proof of Theorem \ref{thm:denom} can be found in Section \ref{sec:pf_main}.
 \begin{theorem}\label{thm:denom}
Recall $F = (f(X_1), \ldots, f(X_n))^{\T}, F_0 = (f_0(X_1), \ldots, f_0(X_n))^{\T}$ and $X_1, \ldots, X_n$ are independently and identically distributed according to the density $\rho$. For $\mu \in \mb R^n$, let $p_{n, \mu}(\cdot)$ denote the $\Gauss_n(\mu, \mr I_n)$ density. Let $\Pi$ be a prior on $L^2_{\rho}$ and $\tilde{\epsilon}_n \to 0$ be a sequence such that $n \tilde{\epsilon}_n^2 \to \infty$. Then, 
\begin{align}\label{eq:denom_bd}
\bbP_0 \bigg( \int \frac{p_{n, F}(Y)}{p_{n, F_0}(Y)} \Pi(df) \ge e^{-n \tilde{\epsilon}_n^2} \Pi(f: \norm{f - f_0}_{2, \rho} < \tilde{\epsilon}_n) \bigg) \ge 1 - C \frac{\log(n \tilde{\epsilon}_n^2)}{\sqrt{n \tilde{\epsilon}_n^2}}.
\end{align} 
\end{theorem}
Using Theorem \ref{thm:denom} along with a standard argument (see, for example, Theorem 2.1 of \cite{ghosal2000convergence}), we can bound
\begin{align}\label{eq:t2n_bd}
T_{2n} \le \frac{\Pi (\norm{f - f^t}_{2, \rho}^2 > M^2\epsilon_n^2)}{e^{- n \tilde{\epsilon}_n^2} \Pi(\norm{f - f_0}_{2, \rho} \leq \tilde{\epsilon}_n)} + C \frac{ \log(n \tilde{\epsilon}_n^2) }{ \sqrt{n \tilde{\epsilon}_n^2} }.
\end{align}
Using Theorem \ref{thm:rd} and Theorem \ref{thm:denom}, we arrive at the following corollary to Theorem \ref{thm:rd_gp_gen}.
\begin{corollary}\label{thm:rd_gp}
Consider model \eqref{eq:npreg} with a GP prior $f \sim \mbox{GP}(0, \sigma^2 K)$. Assume the true function $f_0$ satisfies {\bf (T1)}. Let $\epsilon_n \to 0$ satisfy $n \epsilon_n^2 \to \infty$. Let $k_n < n$ be such that 
\begin{description}
\item[(C0)] The eigenfunctions $\{\phi_j\}_{j=1}^{k_n}$ and eigenvalues $\{\lambda_j\}_{j=1}^{k_n}$ of the kernel $K$ with respect to the covariate density $\rho$ satisfy {\bf (A1)} and {\bf (A2)}.

\item [(C1)] $\max\{k_n, \norm{\theta_0^t}_{\mb H}^2\} = o(n \epsilon_n^2)$.

\item [(C2)] $\norm{f_0 - f_0^t}_{2, \rho}^2 = o(\epsilon_n^2)$. 

\item [(C3)] There exists a sequence $\tilde{\epsilon}_n \to 0$ with $n \tilde{\epsilon}_n^2 \to \infty$ such that
\begin{align}\label{eq:prior_rat1}
\frac{\Pi (\norm{f - f^t}_{2, \rho}^2 > M^2\epsilon_n^2)}{e^{- n \tilde{\epsilon}_n^2} \Pi(\norm{f - f_0}_{2, \rho} \leq \tilde{\epsilon}_n)} \to 0.
\end{align}
\end{description}
Then, for a large constant $M > 0$, 
\begin{align}\label{eq:rates_gp}
\lim_{n \to \infty} \bbE_0 \Pi(\norm{f - f_0}_{2, \rho} > M \epsilon_n \mid Y, X) = 0.
\end{align}
\end{corollary}
\begin{proof}
The quantity in \eqref{eq:rates_gp} is bounded by $T_{1n} + T_{2n}$ from Theorem \ref{thm:rd_gp_gen}. Invoking Theorem \ref{thm:rd}, 
$$
T_{1n} \lesssim \frac{ \max\{k_n, \norm{\theta_0^t}_{\mb H}^2 \} }{n\epsilon_n^2} + \frac{\norm{f_0 - f_0^t}_{2, \rho}^2}{\epsilon_n^2} + P(\chi_{k_n}^2 > M^2 n \epsilon_n^2/4) + \bbP_X(A_n^c).
$$
The first two quantities in the above display converge to zero by {\bf (C1)} and {\bf (C2)}. By {\bf (C1)} and a standard deviation inequality for chi-square distributions, $P(\chi_{k_n}^2 > M^2 n \epsilon_n^2/4) \to 0$ for $M > 2$. 
By Lemma \ref{lem:well_cond}, $\bbP_X(A_n^c) \le k_n e^{-C n/(k_n L_n^2)} \to 0$ by {\bf (C0)}. The proof is completed using the bound \eqref{eq:t2n_bd} for $T_{2n}$.
\end{proof}
In \eqref{eq:prior_rat1}, the prior tail probability in the numerator $\Pi(\norm{f - f^t}_{2, \rho}^2 > M^2\epsilon_n^2) = \Pi(\sum_{j=k_n+1}^{\infty} \lambda_j Z_j^2 > M^2 \epsilon_n^2)$, with $Z_j$s i.i.d. $\Gauss(0, 1)$. Using a version of Bernstein's inequality for sub-exponential random variables (Proposition 5.16 of \cite{vershynin2010introduction}), one can suitably bound this probability. Second, the prior concentration in $L^2_{\rho}$ norm in the denominator $ \Pi(\norm{f - f_0}_{2, \rho} \leq \tilde{\epsilon}_n) = \Pi( \norm{\bft - \bft_0}_{\ell_2} \le \tilde{\epsilon}_n)$ with $\theta_j \sim \Gauss(0, \lambda_j)$; this can be bounded from below using Anderson's inequality (Lemma \ref{lem:anderson} in the Appendix). We provide specific illustrations of these arguments for the squared-exponential kernel below.

\section{Application to the squared-exponential kernel}\label{sec:sqexp}

As a non-trivial application of the general results in the previous section, we consider Gaussian process regression with a squared-exponential kernel $K_a(x, x') = \exp(-a^2 \norm{x - x'}^2)$; a popular choice in machine learning applications. It is well-known that the realizations of a GP with squared-exponential kernel are infinitely smooth and hence are not suitable to model rougher functions. It has only been recently understood \cite{van2007bayesian} that the parameter $a$ plays the role of an ``inverse-bandwidth'', and scaling the parameter $a$ with the sample size enables better approximation of rougher functions. \cite{van2007bayesian} motivates this from a rescaling perspective; choosing a large value of $a$ is equivalent to tracing the trajectory of a smooth process (with $a = 1$) over a larger domain, incurring more roughness. In the regression context \eqref{eq:npreg}, \cite{van2007bayesian} derived optimal posterior  convergence rates in the empirical $L^2$ norm using a rescaling $a\equiv a_n =  n^{1/(2\alpha +1)}$ where the true function is $\alpha$-smooth on a compact domain in $\mathbb{R}$.  Using a gamma prior on $a$, \cite{van2009adaptive} extended their result showing that the rate of contraction is adaptive over any $\alpha$-smooth compactly supported function.  In a more recent article, \cite{pati2015rd} extended the results in  \cite{van2007bayesian}  for integrated $L^1$ norm.  All these articles make exclusive use of the reproducing kernel Hilbert space theory from \cite{van2008reproducing} and bounds on sup-norm small-ball probabilities of Gaussian processes over compact domain \cite{kuelbs1995small,kuelbs1993metric,li1999approximation}. 

\begin{figure}[h]
\hspace{-0.35in}
\includegraphics[width=3.7in]{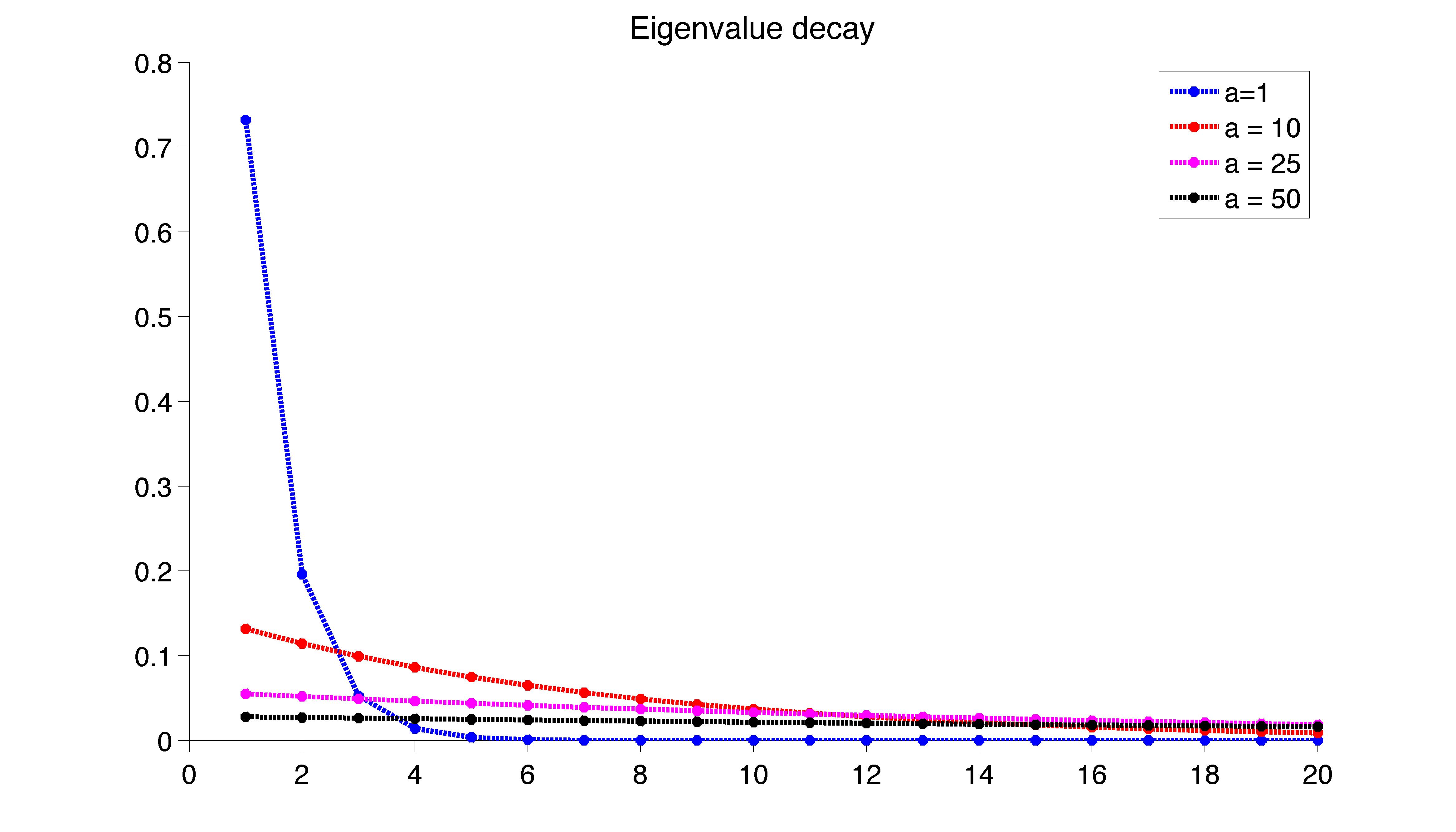}\hspace{-0.35in} 
\includegraphics[width=3.7in]{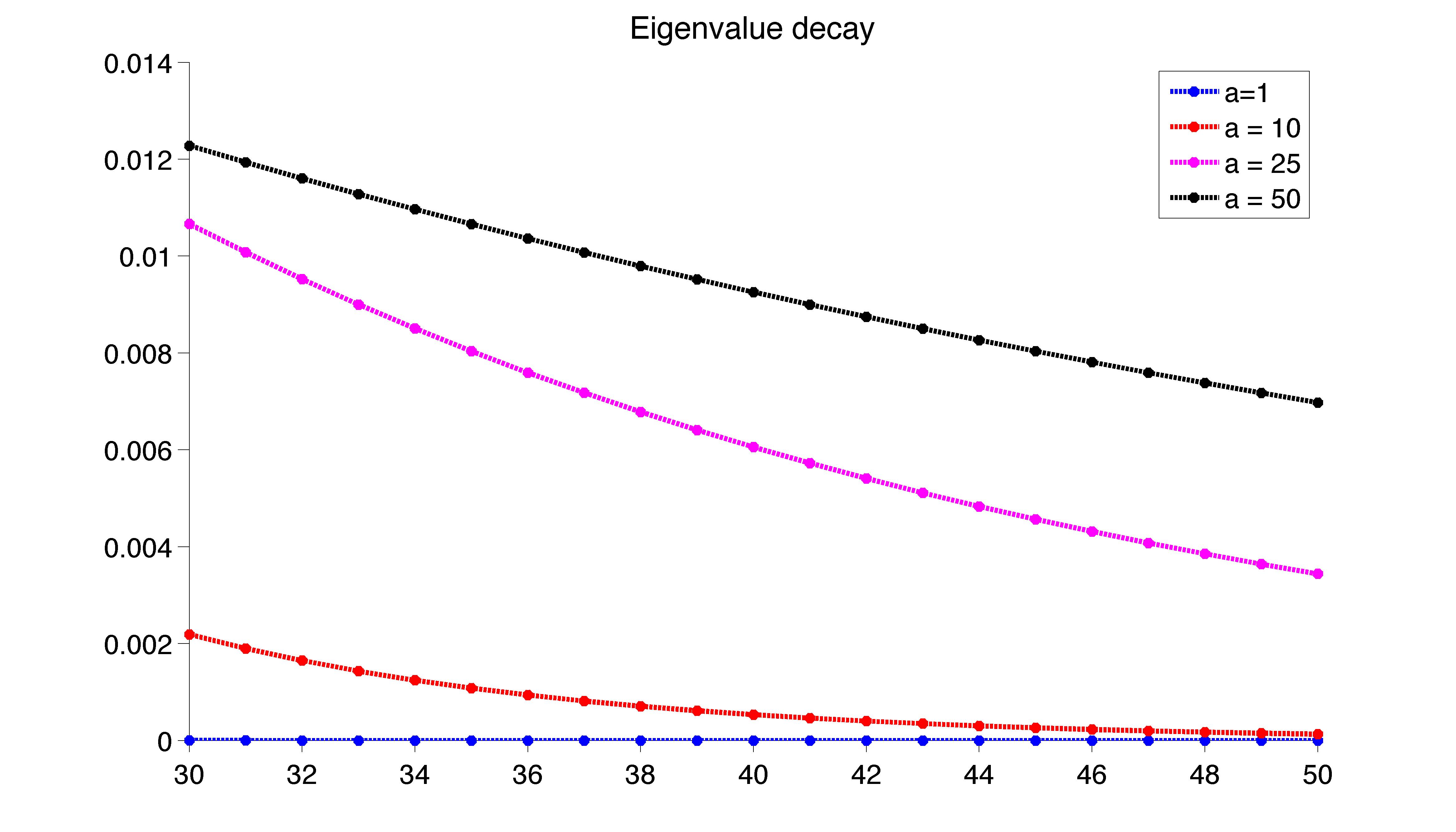}
\caption{The top $50$ eigenvalues $\lambda_j$ of the squared-exponential kernel in \eqref{eq:hermitebasis} plotted against the index $j$ for $4$ different values of $a$. Left panel: the index $j$ runs from 1 to 20. Right panel: $j$ runs from 21 to 50. With increasing $a$, the rate of decay is slowed down.} \label{fig:eigdec}
\end{figure}
The eigen-expansion of the squared-exponential kernel offers a complementary perspective into the rescaling phenomenon. Consider the expression for the eigenvalues of the squared-exponential kernel in \eqref{eq:hermitebasis}. It is well known that the rate of decay of the eigenvalues is closely connected to the smoothness of the process \eqref{eq:kar_lo}. When $a = 1$, the eigenvalues $\lambda_j$ decay exponentially fast in $j$, indicating the infinite smoothness of the sample paths. Although the rate of decay remains exponential in $j$ for any fixed value of $a$, it is effectively slowed down for large values of $a$; see Figure \ref{fig:eigdec} for an illustration. 

In this section, we apply the results developed in Section 4 (specifically Corollary \ref{thm:rd_gp}) to derive posterior rates of contraction for the above rescaled GP priors with the covariates drawn i.i.d. from a Gaussian density on the real line. To best of our knowledge, no existing posterior contraction rate result for the squared-exponential (or other) kernel allows unbounded covariate support. 
Using a tensor-product basis approach, it is possible to extend our results to covariates in $\mb R^d$.

\subsection{{\bf Posterior contraction rates}}

For the remainder of this Section, $\{\phi_j\}$ and $\{\lambda_j\}$ denote the eigenfunctions and eigenvalues \eqref{eq:hermitebasis} of the squared-exponential kernel with inverse-bandwidth parameter $a$; the dependence on $a$ is suppressed for notational convenience. In order to apply Theorem \ref{thm:rd_gp} to the squared-exponential kernel, we need sup-norm bounds on the $k_n$ leading eigenfunctions $\phi_j$s. Since we are concerned with rescaled processes where the parameter $a$ is sample-size dependent, it is important to precisely characterize the role of $a$ in the bound. 

A well-known inequality for the Hermite polynomial is Cramer's bound \cite{gabor1939orthogonal}, which states that for any $l \ge 1$, $|H_l(z)| \le C \sqrt{2^l \, l!} \, e^{z^2/2}$ for all $z \in \mb R$, where $C \le 2$ is a global constant which doesn't depend on $z$ or $l$. A direct use of this bound leads to $|\phi_{j+1}(x)| \lesssim (c/b)^{1/4} e^{b x^2}$, which is clearly not sufficient as we are dealing with unbounded covariates. Since the Hermite functions are polynomials, the exponential bound provided by Cramer's inequality is wasteful in the tails. We derive a bound for the leading eigenfunctions $\phi_j$s in Lemma \ref{lem:eigbd} below; refer to the Appendix for a proof. We did not find an existing reference proving this result. The main idea is to use Cramer's bound in a neighborhood of the origin, while for suitably large values of $x$, use a combination of Cramer's bound with a different bound obtained by exploiting an integral representation of the Hermite polynomials. 
\begin{lemma}\label{lem:eigbd}
Let $\phi_l$s be the eigenfunctions of the squared-exponential kernel as in \eqref{eq:hermitebasis}. Then, 
$\max_{0 \le j \le k} \sup_{x \in \mb R}|\phi_{j+1}(x)| \lesssim  a^{1/4} e^{b k/a}$ for large $a$. 
\end{lemma}


We are now in a position to state the rate theorem. 
Set $a_n = n^{1/(2 \alpha + 1)}$ in \eqref{eq:hermitebasis}. We define the true class of functions $\m F$ with ``smoothness $\alpha$'' as linear combinations of the eigenfunctions $\phi_j$ with the coefficient vector in the Sobolev class $\Theta_{\alpha}$. Formally,  
\begin{align}\label{eq:F0}
\m F = \{f_0 : f_0 = \sum_{j=1}^{\infty} \theta_{0j} \phi_j, \, \bft_0 = (\theta_{01}, \theta_{02}, \ldots ) \in \Theta_{\alpha} \}.
\end{align}
\begin{theorem}\label{thm:rd_gpsqexp}
Consider the nonparametric regression model \eqref{eq:npreg}. Assume the covariates $X_i$ are drawn i.i.d. from a Gaussian density $\rho(x) = \sqrt{2b/\pi} \, e^{-2 b x^2}$ and the true function $f_0 \in \m F$ as in \eqref{eq:F0} with $\alpha > 1/\{4(1 - 2b)\}$. Let $f \sim \mbox{GP}(0, \sigma^2 K)$ with squared-exponential covariance kernel $K_a(x, x') = \exp(-a^2 |x - x'|^2)$. Choose $a \equiv a_n = n^{1/(2 \alpha + 1)}$. Then, an upper bound to the posterior contraction rate \eqref{eq:rates_gp} in $L^2_{\rho}$ norm is $\epsilon_n = n^{-\alpha/(2 \alpha + 1)} \log n$. 
\end{theorem}
\begin{rem}
From \cite{van2007bayesian}, the rescaling $a_n = n^{1/(2 \alpha + 1)}$ is the optimal choice for $\alpha$ smooth functions on a compact domain and leads to the optimal rate $n^{-\alpha/(2 \alpha + 1)}$ up to a logarithmic term. Theorem \ref{thm:rd_gpsqexp} obtains a similar result for non-compact domains in a random design setting. The lower bound on the smoothness $\alpha$ is typically necessitated in random design settings; see for example, \cite{brown2002asymptotic,birge2004model}. 
In particular, when $b = 1/4$, so that $\rho$ corresponds to the standard normal density, we require $\alpha > 1/\{4(1 - 2b)\} = 1/2$. 
\end{rem}



\section{Proof of main results}\label{sec:pf_main}

\subsection*{Proof of Theorem \ref{thm:rd_gp_gen}}

Using triangle inequality $\norm{f - f_0}_{2, \rho} \le \norm{f^t -f_0^t}_{2, \rho} + \norm{f - f^t}_{2, \rho} + \norm{f_0 - f_0^t}_{2, \rho}$, and since $\norm{f_0 - f_0^t}_{2, \rho} \lesssim \epsilon_n$ by assumption, we can bound $\Pi(\norm{f - f_0}_{2, \rho} > M \epsilon_n \mid Y, X) \le \Pi(\norm{f^t - f_0^t}_{2, \rho} > M \epsilon_n \mid Y, X) + \Pi(\norm{f - f^t}_{2, \rho} > M \epsilon_n \mid Y, X)$.  Further, using the orthonormality of the eigenfunctions, $\Pi(\norm{f^t - f_0^t}_{2, \rho} > M \epsilon_n \mid Y, X) = \wt{\m W}^t(\norm{\theta^t - \theta_0^t} > M \epsilon_n \mid Y, X)$. Therefore, taking expectation, 
\begin{align}\label{eq:basic_post}
\bbE_0 \Pi(\norm{f - f_0}_{2, \rho} > M \epsilon_n \mid Y, X) \le \bbE_0 \wt{\m W}^t(\norm{\theta^t - \theta_0^t} > M \epsilon_n \mid Y, X) + T_{2n}.
\end{align}

Let $U_n = \{ \norm{\theta^t - \theta_0^t} \le M \epsilon_n\}$. We shall show below that $\bbE_0 \wt{\m W}^t(U_n^c \mid Y, X) \le T_{1n}$, which will complete the proof of the theorem. 
For any $A_n \subset \m X^n$ in the $\sigma$-field generated by $X_1, \ldots, X_n$, bound 
\begin{align}\label{eq:indicator}
\bbE_0 \wt{\m W}^t(U_n^c \mid Y, X) \le \bbE_0 \wt{\m W}^t(U_n^c \mid Y, X) \, \ind_{A_n}(X) + \bbP_X(A_n^c),
\end{align}
We now elucidate a coupling argument to bound the $\wt{\m W}^t(U_n^c \mid Y, X)$ term in \eqref{eq:indicator}.
Given $(Y, X)$, let $(\theta_T, \theta_A) \in \mb R^{k_n} \otimes \mb R^{k_n}$ be a pair of random variables such that $\theta_T \sim Q_T \equiv \wt{\m W}^t(\cdot \mid Y, X), \theta_A \sim Q_A \equiv \Gauss(\theta_0^t, \sigma^2 \mr I_{k_n}/n)$ and $E \norm{\theta_T - \theta_A}^2 = d_{W,2}^2(Q_T, Q_A)$, where $E$ denotes an expectation with respect to the joint distribution of $(\theta_T, \theta_A)$ given $Y, X$. In other words, $(\theta_T, \theta_A) \in \mbox{joint}(Q_T, Q_A)$ are {\em optimally coupled}, i.e., the infimum in \eqref{eq:wstein_def} is attained by $(\theta_T, \theta_A)$. 
Such an optimal coupling can be always constructed in general; see \cite{givens1984class} for a a constructive proof for normal distributions. 
We then have 
\begin{align}
\wt{\m W}^t(U_n^c \mid Y, X) 
& = P(\theta_T \in U_n^c) \notag \\
& \le P(\theta_T \in U_n^c, \norm{\theta_T - \theta_A} \le M \epsilon_n/2) + P(\norm{\theta_T - \theta_A} > M \epsilon_n/2) \label{eq:bd_int1} \\
& \le P(\norm{\theta_A - \theta_0^t} > M \epsilon_n/2) + \frac{4 E \norm{\theta_T - \theta_A}^2}{M^2 \epsilon_n^2} \notag \\
& = P(\chi_{k_n}^2 > M^2 n \epsilon_n^2/4) + \frac{4 d_{W,2}^2(Q_T, Q_A)}{M^2 \epsilon_n^2} \label{eq:coupling_bd}.
\end{align}
In the above display, the first line simply uses that the marginal distribution of $\theta_T$ is $\wt{\m W}^t(\cdot \mid Y, X)$ by construction. From the first to the second line \eqref{eq:bd_int1}, we use a union bound. 
For the first term in \eqref{eq:bd_int1}, we first use triangle inequality to conclude that if $\theta_T \in U_n^c$, i.e., $\norm{\theta_T - \theta_0^t} > M \epsilon_n$, and $\norm{\theta_T - \theta_A} \le M \epsilon_n/2$, then $\norm{\theta_A - \theta_0^t} \ge \norm{\theta_T - \theta_0^t} - \norm{\theta_T - \theta_A} \ge M \epsilon_n/2$. Next, by construction, $(\theta_A - \theta_0^t) \mid Y, X \sim \Gauss(0, \sigma^2/n \mr I_{k_n})$, which implies $P(\norm{\theta_A - \theta_0^t} > M \epsilon_n/2) = P(\chi_{k_n}^2 > M^2 n \epsilon_n^2/4)$. 
For the $P(\norm{\theta_A - \theta_0^t} > M \epsilon_n/2)$ term in \eqref{eq:bd_int1}, we first use Markov's inequality, and then exploit the fact that $(\theta_T, \theta_A)$ are ``optimally coupled'', i.e., $E \norm{\theta_T - \theta_A}^2 = d_{W,2}^2(Q_T, Q_A)$. This leaves us at \eqref{eq:coupling_bd}. Finally, substituting the bound \eqref{eq:coupling_bd} in \eqref{eq:indicator}, we have 
{\small$$
\bbE_0 \wt{\m W}^t(U_n^c \mid Y, X) \le \frac{\bbE_0\left\{ \ind_{A_n}(X) \d_{W,2}^2 \bigg[\wt{\m W}^t(\cdot \mid Y, X), \Gauss_{k_n}\bigg(\theta_0^t, \frac{\sigma^2}{n} \mr I_{k_n}\bigg)  \bigg] \right\} }{M^2 \epsilon_n^2/4} + P(\chi_{k_n}^2 > M^2 n \epsilon_n^2/4)  + \bbP_X(A_n^c).
$$}
The quantity in the right hand side in the above display is $T_{1n}$, and the theorem is proved.

\subsection*{{\bf Proof of Lemma \ref{lem:well_cond} \& Remark \ref{rem:well_cond}}}
We make use of the following version of a matrix Bernstein inequality from \cite{tropp2012user}: let $Z_i, i = 1, \ldots, n$ be a sequence of independent self-adjoint $d \times d$ matrices with $\bbE Z_i = 0$ and $\| Z_i \|_2 \leq B$ almost surely for some $B > 0$. Let $\eta^2 = \norm{ \sum_{i=1}^n \bbE Z_i^2}_2$. Then, for any $t > 0$, 
\begin{align}\label{eq:matbern_bded}
\bbP \big( \| \sum_{i=1}^n Z_i \| > t \big) \leq d \exp \bigg(- \frac{t^2/2}{\eta^2 + Bt/3} \bigg). 
\end{align} 
Set $\phi^{(i)} = (\phi_j(X_i))_{1\le j \le k_n} \in \mb R^{k_n}$ and $Z_i = \phi^{(i)} (\phi^{(i)})^{\T} - \mr I_{k_n}$, so that $\sum_{i=1}^n Z_i = \Phi^{T} \Phi - n \mr I_{k_n}$. The $Z_i$s are independent symmetric matrices with $\bbE_X Z_i = 0$, since from the orthonormality of the eiegnfunctions $\{\phi_j\}$, $(\bbE_X \phi^{(i)} (\phi^{(i)})^{\T})_{jl} = \int \phi_j(x) \phi_l(x) \rho(x) dx = \delta_{jl}$. We also have $\norm{Z_i}_2 \leq 1 + \norm{\phi^{(i)}}_2^2 = 1 + \sum_{j=1}^{k_n} |\phi_j^2(X_i)| \leq 1 + k_n L_n^2 \lesssim k_n L_n^2$. Therefore, the conditions for applying \eqref{eq:matbern_bded} are satisfied. 

We have $Z_i^2 = \norm{\phi^{(i)}}^2 \phi^{(i)} (\phi^{(i)})^{\T} - 2 \phi^{(i)} (\phi^{(i)})^{\T} + \mr I_{k_n} \prec \norm{\phi^{(i)}}_2^2  \phi^{(i)} (\phi^{(i)})^{\T} + \mr I_{k_n} \prec k_n L_n^2 \phi^{(i)} (\phi^{(i)})^{\T} + \mr I_{k_n}$, so that $\norm{ \bbE Z_i^2} \lesssim k_n L_n^2$ and hence by triangle inequality, $\eta^2 \lesssim n k_n L_n^2$. Substituting $t = n/2$ and $B = k_n L_n^2$ in \eqref{eq:matbern_bded}, we have
$$
\bbP_X( \norm{ \Phi^{\T} \Phi - n \mr I_{k_n} } > n/2 ) \le k_n \exp \bigg(- \frac{C n^2}{\eta^2 + Bn/6} \bigg) \le k_n e^{- C n/(k_n L_n^2)},
$$
since $\eta^2 + B n/6 \le n k_n L_n^2 + n k_n L_n^2/6 \le C n k_n L_n^2$ and $e^{-1/x}$ is increasing in $x$. 

Remark \ref{rem:well_cond} follows, since on $A_n$, \\
\noindent (i) using triangle inequality, $\norm{\Phi^{\T} \Phi}_2 \le 3n/2$. \\
\noindent (ii) using Lemma A.1 (ii), $s_{\min}(\Phi^{\T} \Phi) \ge n - \norm{\Phi^{\T} \Phi - n \mr I_{k_n}} \ge n/2$. \\
\noindent (iii) using (i) and (ii), $\kappa(\Phi^{\T} \Phi) \le 3$. \\
\noindent (iv) $\tr[(\Phi^{\T} \Phi)^{-1}] \le k_n \norm{(\Phi^{\T} \Phi)^{-1}}_2  = k_n/s_{\min}(\Phi^{\T} \Phi) \le 2k_n/n$.

\subsection*{{\bf Proof of Theorem \ref{thm:rd} }}
Given $Y, X$, recall that $Q_T$ and $Q_A$ respectively denote the probability measures $\wt{\m W}(\cdot \mid Y, X) \equiv \Gauss_{k_n}(\wt{\theta}, \wt{\Sigma})$ and $\Gauss_{k_n}(\theta_0^t, \sigma^2 \mr I_{k_n}/n)$. By the tower property of conditional expectation, 
\begin{align}\label{eq:tower}
\bbE_0 \big[ \ind_{A_n}(X) d_{W, 2}^2(Q_T, Q_A) \big] = \bbE_X \big[\ind_{A_n}(X) \, \bbE_{0 \mid X} d_{W, 2}^2(Q_T, Q_A) \big].
\end{align}
Since $\wt{\Sigma}$ and $\sigma^2 \mr I_{k_n}/n$ (trivially) commute, apply \eqref{eq:wstein_normal} to write
\begin{align}\label{eq:wstein_app}
d_{W,2}^2 (Q_T, Q_A) = \norm{ \wt{\theta} - \theta_0^t}^2 + \norm{ \wt{\Sigma}^{1/2} - \frac{\sigma}{\sqrt{n}} \mr I_{k_n}}_F^2.
\end{align}
Thus, 
\begin{align}\label{eq:cond_exp}
\bbE_{0 \mid X} d_{W, 2}^2(Q_T, Q_A) = \bbE_{0 \mid X} \norm{ \wt{\theta} - \theta_0^t}^2 + \norm{ \wt{\Sigma}^{1/2} - \frac{\sigma}{\sqrt{n}} \mr I_{k_n}}_F^2,
\end{align}
since the second term does not involve $Y$. We now proceed to bound each of these two terms in \eqref{eq:cond_exp} on the set $A_n$. To that end, we shall apply Lemma \ref{lem:well_cond} and in particular, the consequences of Lemma \ref{lem:well_cond} summarized in Remark \ref{rem:well_cond} multiple times below. We also make use of Lemma \ref{lem:tprod_frob} on multiple occasions. 

Recall $\wt{\theta} = (\Phi^{\T} \Phi + \Lambda^{-1})^{-1}$ and define $\theta_Y = (\Phi^{\T} \Phi)^{-1} \Phi^{\T} Y$. Using $\norm{a+b}^2 \le 2 (\norm{a}^2 + \norm{b}^2)$, bound 
\begin{align}\label{eq:w2_t1}
\bbE_{0 \mid X} (\wt{\theta} - \theta_0^t)^2 \lesssim \bbE_{0 \mid X} \norm{\wt{\theta} - \theta_Y}^2 + \bbE_{0 \mid X}\norm{\theta_Y - \theta_0^t}^2.
\end{align} 
Let us first deal with $\bbE_{0 \mid X}\norm{\theta_Y - \theta_0^t}^2$. Let $F_0 = (f_0(X_1), \ldots, f_0(X_n))^{\T}$, so that $F_0 = \bbE_{0 \mid X} Y$. By {\bf (T1)}, we can write $F_0 = \Phi \theta_0^t + R$, where $R = (R_1, \ldots, R_n)^{\T}$ with $R_i = f_0(X_i) - f_0^t(X_i)$. Write $\bbE_{0 \mid X} \norm{\theta_Y - \theta_0^t}^2 = \bbE_{0 \mid X} \norm{\theta_Y - \bbE_{0 \mid X} \theta_Y}^2 + \norm{\bbE_{0 \mid X} \theta_Y - \theta_0^t}^2$. The first term $\bbE_{0 \mid X} \norm{\theta_Y - \bbE_{0 \mid X} \theta_Y}^2 = \sigma^2 \tr[(\Phi^{\T} \Phi)^{-1}] \lesssim \sigma^2 k_n/n$ on $A_n$. For the second term, write $\bbE_{0 \mid X} \theta_Y = (\Phi^{\T} \Phi)^{-1} \Phi^{\T} F_0 = \theta_0^t + (\Phi^{\T} \Phi)^{-1} \Phi^{\T} R$, so that $ \norm{\bbE_{0 \mid X} \theta_Y - \theta_0^t}^2 = \norm{(\Phi^{\T} \Phi)^{-1} \Phi^{\T} R }^2 \le \norm{(\Phi^{\T} \Phi)^{-1} \Phi^{\T}}_2^2 \norm{R}^2$. Finally, $\norm{(\Phi^{\T} \Phi)^{-1} \Phi^{\T}}_2 \le \norm{\Phi^{\T}}_2/s_{\min}(\Phi^{\T} \Phi) = \{\norm{\Phi^{\T} \Phi}_2\}^{1/2}/s_{\min}(\Phi^{\T} \Phi)$, and the last quantity is bounded above by a constant multiple of $1/\sqrt{n}$ on $A_n$. Therefore, 
\begin{align}\label{eq:ek_nombor}
\ind_{A_n}(X) \bbE_{0 \mid X} \norm{\theta_Y - \theta_0^t}^2 \lesssim \sigma^2 k_n/n + \norm{R}^2/n.
\end{align}

We now handle the $\bbE_{0 \mid X} \norm{\wt{\theta} - \theta_Y}^2$ term in \eqref{eq:w2_t1}. Using $A_1^{-1} - A_2^{-1} = A_1^{-1} (A_2 - A_1) A_2^{-1}$, write $\theta_Y - \wt{\theta} = BY$, where $B = (\Phi^{\T} \Phi)^{-1} \Delta \Phi^{\T}$ with $\Delta = \Lambda^{-1} (\Phi^{\T} \Phi + \Lambda^{-1})^{-1}$. Using a standard result for expectations of quadratic forms, 
\begin{align}\label{eq:w2_t1a}
\bbE_{0 \mid X} \norm{\wt{\theta} - \theta_Y}^2 = \norm{B F_0 }^2 + \sigma^2 \norm{B}_F^2.
\end{align} 
So the goal now is to bound each one of $\norm{B F_0}^2$ and $\norm{B}_F^2$ on $A_n$. We have $B F_0 = (\Phi^{\T} \Phi)^{-1} \Delta (\Phi^{\T} \Phi) \theta_0^t + (\Phi^{\T} \Phi)^{-1} \Delta \Phi^{\T} R = (\Phi^{\T} \Phi)^{-1} \Lambda^{-1} \theta_0^s + (\Phi^{\T} \Phi)^{-1} \Delta \Phi^{\T} R$, where $\theta_0^s = (\Phi^{\T} \Phi + \Lambda^{-1})^{-1} (\Phi^{\T} \Phi) \theta_0^t$. Bound $\norm{B F_0}^2 \le 2 \big( \norm{(\Phi^{\T} \Phi)^{-1} \Lambda^{-1} \theta_0^s}^2 + \norm{(\Phi^{\T} \Phi)^{-1} \Delta \Phi^{\T} R}^2 \big)$. Bound $\norm{\Delta}_2 \le \norm{\Lambda^{-1}}_2/\{s_{\min}(\Phi^{\T} \Phi) - \norm{\Lambda^{-1}}_2\} \lesssim 1$ on $A_n$, since $\norm{\Lambda^{-1}}_2 < n/4$ by {\bf (A1)}. Therefore, $\norm{(\Phi^{\T} \Phi)^{-1} \Delta \Phi^{\T} R} \lesssim \norm{\Phi^{\T}}/s_{\min}(\Phi^{\T} \Phi) \norm{R} \lesssim \norm{R}/\sqrt{n}$ on $A_n$, using an argument as in the paragraph after the display \eqref{eq:w2_t1}. Next, $\norm{(\Phi^{\T} \Phi)^{-1} \Lambda^{-1} \theta_0^s} \le \big\{ \norm{\Lambda^{-1/2}}/s_{\min}(\Phi^{\T} \Phi) \big\}\norm{\Lambda^{-1/2} \theta_0^s} \le \norm{\Lambda^{-1/2} \theta_0^s}/\sqrt{n}$ on $A_n$. After some manipulation, we can write $\theta_0^s = \theta_0^t - \Delta^{\T} \theta_0^t$, so that 
\begin{align*}
& \norm{\Lambda^{-1/2} \theta_0^s} \le \norm{\Lambda^{-1/2} \theta_0^t} + \norm{\Lambda^{-1/2} \Delta^{\T} \theta_0^t} \\
& \le \norm{\Lambda^{-1/2} \theta_0^t} + \norm{\Lambda^{-1/2} (\Phi^{\T} \Phi + \Lambda^{-1})^{-1} \Lambda^{-1/2}}_2 \norm{\Lambda^{-1/2} \theta_0^t} \lesssim \norm{\Lambda^{-1/2} \theta_0^t} = \norm{\theta_0^t}_{\mb H},
\end{align*}
since $\norm{\Lambda^{-1/2} (\Phi^{\T} \Phi + \Lambda^{-1})^{-1} \Lambda^{-1/2}}_2 = \norm{\Delta}_2$, which we already know is $\lesssim 1$ on $A_n$. 
Thus, we conclude that $\norm{B F_0}^2 \lesssim \norm{R}^2/n + \norm{\theta_0^t}_{\mb H}^2/n$ on $A_n$. Finally, 
$$
\norm{B}_F^2 \le k_n \norm{B}_2^2 \le \frac{k_n \norm{\Delta}_2^2 \norm{\Phi^{\T}}_2^2}{s_{\min}^2(\Phi^{\T} \Phi)} \lesssim \frac{k_n}{n}
$$
on $A_n$, since we have already shown that $\norm{\Delta}_2 \lesssim 1$ and $\norm{\Phi^{\T}}_2/s_{\min}(\Phi^{\T} \Phi) \lesssim 1/\sqrt{n}$ on $A_n$. Substituting all the inequalities in \eqref{eq:w2_t1a},  
\begin{align}\label{eq:dui_nombor}
\ind_{A_n}(X) \bbE_{0 \mid X} \norm{\wt{\theta} - \theta_Y}^2 \lesssim \norm{R}^2/n + \norm{\theta_0^t}_{\mb H}^2/n + \sigma^2 k_n/n.
\end{align}
Substituting the inequalities obtained in \eqref{eq:ek_nombor} and \eqref{eq:dui_nombor} in \eqref{eq:w2_t1}, 
\begin{align}\label{eq:w2_t1_f}
\ind_{A_n}(X) \bbE_{0 \mid X} (\wt{\theta} - \theta_0^t)^2 \lesssim \sigma^2 \frac{k_n}{n} + \frac{\norm{\theta_0^t}_{\mb H}^2}{n} + \frac{\norm{R}^2}{n}.
\end{align}

\medskip
Now we consider the term $\norm{ \wt{\Sigma}^{1/2} - \frac{\sigma}{\sqrt{n}} \mr I_{k_n}}_F^2$ in \eqref{eq:cond_exp}. Recalling the expression of $\wt{\Sigma}$, $\wt{\Sigma}^{1/2} = \sigma (\Phi^{\T} \Phi + \Lambda^{-1})^{-1/2}$, and since $n/4 \le s_{\min}(\Phi^{\T} \Phi + \Lambda^{-1}) \le \norm{\Phi^{\T} \Phi + \Lambda^{-1}} \le 2n$ on $A_n$, all eigenvalues of $\wt{\Sigma}$ are of the form $C \sigma/\sqrt{n}$ on $A_n$. Since the squared Frobenius norm of a matrix is the sum of the squared eigenvalues, we conclude that $\norm{ \wt{\Sigma}^{1/2} - \frac{\sigma}{\sqrt{n}} \mr I_{k_n}}_F^2 \lesssim \sigma^2 k_n/n$ on $A_n$. This, in conjunction with \eqref{eq:w2_t1_f}, when substituted in \eqref{eq:cond_exp} yield
\begin{align}\label{eq:cond_exp2}
\ind_{A_n}(X) \, \bbE_{0 \mid X} d_{W, 2}^2 (Q_T, Q_A) \le \sigma^2 \frac{k_n}{n} + \frac{\norm{\theta_0^t}_{\mb H}^2}{n} + \frac{\norm{R}^2}{n}.
\end{align} 
Recall from \eqref{eq:tower} that our objective is bound the $\bbE_X$ expectation of the left hand side of \eqref{eq:cond_exp2}. The only term depending on $X$ in the right hand side of \eqref{eq:cond_exp2} is $\norm{R}^2$ and $\bbE_X \norm{R}^2 = n \norm{f_0 - f_0^t}_{2, \rho}^2$. Therefore, taking an expectation with respect to $\bbE_X$ on both sides of \eqref{eq:cond_exp2}, the conclusion follows.


\subsection*{{\bf Proof of Theorem \ref{thm:denom} }}
Let 
\begin{align*}
D_n = \int \frac{p_{n, F}(Y)}{p_{n, F_0}(Y)} \Pi(df), \quad G_n = \int \log \bigg\{ \frac{p_{n, F}(Y)}{p_{n, F_0}(Y)}  \bigg\}\Pi(df).
\end{align*}
Following a standard argument, it is enough to show the desired lower bound on $\bbP_0(D_n \ge e^{- n \tilde{\epsilon}_n'^2})$ for any probability measure $\Pi$ supported on $\m F_n = \{ f : \norm{f - f_0}_{2, \rho} < \tilde{\epsilon}_n\}$. By Jensen's inequality, $\log D_n \ge G_n$, so that $\bbP_0(D_n \ge e^{- n \tilde{\epsilon}_n^2}) \ge \bbP_0(G_n \ge - n \tilde{\epsilon}_n^2)$. Our goal below is to bound $\bbP_0(G_n \ge - n \tilde{\epsilon}_n^2)$ from below, or equivalently, bound $\bbP_0(G_n \le - n \tilde{\epsilon}_n^2)$ from above.  

A simple calculation yields $G_n = \mu_{0X}^{\T} (Y - F_0) - \sigma_{0X}^2/2$, where $\mu_{0X}  = \int (F - F_0) \Pi(df) \in \mb R^n$ and $\sigma_{0X}^2 = \int \norm{F - F_0}^2 \Pi(df)$. 
Since $Y \sim \Gauss(F_0, \mr I_n)$, we have $G_n \mid X \sim \Gauss(-\sigma_{0X}^2/2, \norm{\mu_{0X}}^2)$. Also, the marginal expectation of $G_n$, $\bbE_0 G_n = - \bbE_{0X} \sigma_{0X}^2/2 = - n \sigma_{0}^2/2$, where $\sigma_{0}^2 = \int \norm{f - f_0}_{2, \rho}^2 \Pi(df)$. Since $\Pi$ is supported on $\m F_n$, clearly $\sigma_0^2 \le \tilde{\epsilon}_n^2$. 

The Paley--Zygmund inequality (see, for example, \cite{durrett2010probability}) states that for any non-negative random variable $Z$ with finite second moment and $\delta \in (0, 1)$, 
$P(Z \ge \delta EZ) \ge (1 - \delta)^2 (EZ)^2/(EZ^2)$. 
In particular, if $(EZ)^2/(EZ^2) \ge 1 - \gamma$ for $\gamma > 0$ small, then 
\begin{align}\label{eq:PZ1}
P(Z < \delta EZ) \le 1 - (1 - \delta)^2(1 - \gamma) \lesssim \delta + \gamma.
\end{align}
We shall invoke \eqref{eq:PZ1} with the non-negative random variable $Z_n = e^{t_n G_n}$ for some $t_n \in (0, 1/2)$ and $\delta_n \in (0, 1)$ to be chosen below. A key ingredient of such an exercise is to obtain a lower bound on $(\bbE_0 Z_n)^2/(\bbE_0 Z_n^2)$. 

By Jensen's inequality, $\bbE_0 Z_n \ge e^{t_n \bbE_0 G_n} = e^{-n t_n \sigma_0^2/2}$, which implies $(\bbE_0 Z_n)^2 \ge e^{-n t_n \sigma_0^2}$. We next need to bound $\bbE_0 Z_n^2 = \bbE_0 e^{2 t_n G_n}$ from above. Since $G_n \mid X$ is conditionally Gaussian, we have sufficient control over the moment generating function $M_{G_n}(\lambda) = \bbE_0 e^{\lambda G_n}$ for $\lambda \in (0, 1)$. Using the iterative property of conditional expectations, we can write $\bbE_0 e^{\lambda G_n} = \bbE_{0X} [ \bbE_{0\mid X}(e^{\lambda G_n})]$. Recalling $G_n \mid X \sim \Gauss(-\sigma_{0X}^2/2, \norm{\mu_{0X}}^2)$, we have 
$$
\bbE_{0\mid X}(e^{\lambda G_n}) = e^{- \lambda \sigma_{0X}^2/2} \, e^{\lambda^2 \norm{\mu_{0X}}^2/2} \le e^{-(\lambda - \lambda^2) \sigma_{0X}^2/2}, 
$$
where the second step follows since by an application of Cauchy--Schwartz inequality, $\norm{\mu_{0X}}^2 \le \sigma_{0X}^2$. Since $\lambda \in (0, 1)$, the quantity $\lambda - \lambda^2$ in the exponent is positive. Therefore, by Jensen's inequality, $\bbE_0 e^{\lambda G_n} \le e^{-(\lambda - \lambda^2) \bbE_{0X} \sigma_{0X}^2/2 } = e^{-(\lambda - \lambda^2) n \sigma_0^2/2}$. In particular, for any $t_n \in (0, 1/2)$, 
$\bbE_0 Z_n^2 = \bbE_0 e^{2 t_n G_n} \le e^{-n (t_n - 2 t_n^2) \sigma_0^2}$. Combining this bound with the previously obtained bound $(\bbE_0 Z_n)^2 \ge e^{-n t_n \sigma_0^2}$, we have $(\bbE_0 Z_n)^2/(\bbE_0 Z_n^2) \ge e^{-2 t_n^2 n \sigma_0^2} \ge e^{-2 t_n^2 n \tilde{\epsilon}_n^2}$. 

For a slowly decaying sequence $\gamma_n$ satisfying $\gamma_n \to 0$ and $\gamma_n n \tilde{\epsilon}_n^2 \to \infty$, set $t_n^2 n \tilde{\epsilon}_n^2 = \gamma_n$. For $n$ large enough so that $\gamma_n \le 1$, we have $(\bbE_0 Z_n)^2/(\bbE_0 Z_n^2) \ge e^{-\gamma_n} \ge 1 - \gamma_n$. From \eqref{eq:PZ1}, we therefore have that for any $0 < \delta_n > 1$, $\bbP_0(Z_n \le \delta_n \bbE_0 Z_n) \le \delta_n + \gamma_n$. Further, 
\begin{align}\label{eq:ZP2}
\bbP_0(Z_n < \delta_n \bbE_0 Z_n) &= \bbP_0\bigg(G_n < \frac{\log \delta_n}{t_n} + \frac{\log \bbE_0Z_n}{t_n}\bigg) \nonumber \\
& \ge \bbP_0\bigg(G_n < \frac{\log \delta_n}{t_n}  - \frac{n \tilde{\epsilon}_n^2}{2} \bigg),
\end{align}
where the inequality follows since $(\log \bbE_0 Z_n)/t_n \ge \bbE_0 G_n = - n \sigma_0^2/2 \ge - n \tilde{\epsilon}_n^2/2$. Choose $\delta_n$ so that $(\log \delta_n/t_n) = - n \tilde{\epsilon}_n^2/2$, i.e., $\delta_n = e^{- t_n n \tilde{\epsilon}_n^2/2} = e^{-C \sqrt{\gamma_n n \tilde{\epsilon}_n^2}}$. From \eqref{eq:ZP2} and the immediately preceding inequality, we therefore have 
\begin{align}\label{eq:ZP3}
\bbP_0(G_n \le - n \tilde{\epsilon}_n^2) \le \delta_n + \gamma_n = e^{-C \sqrt{\gamma_n n \tilde{\epsilon}_n^2}} + \gamma_n \le e^{-C \gamma_n \sqrt{n \tilde{\epsilon}_n^2}} + \gamma_n.
\end{align}
The sequence $\gamma_n$ is yet to be chosen; we shall do so now by optimizing the right hand side of \eqref{eq:ZP3}. Consider the function $g(x) = x + e^{-B x}$ for $B > 0$. The function attains its minimum value on $(0, \infty)$ at the point $x = \log B/B$ and the minimum value of the function is $(\log B + 1)/B$. Therefore, choose $\gamma_n = C \log (n \tilde{\epsilon}_n^2)/\sqrt{n \tilde{\epsilon}_n^2}$; note that for this choice $\gamma_n \to 0$ and $\gamma_n n \tilde{\epsilon}_n^2 \to \infty$; with this choice we have $\bbP_0(G_n \le - n \tilde{\epsilon}_n^2) \le C \log (n \tilde{\epsilon}_n^2)/\sqrt{n \tilde{\epsilon}_n^2}$.  

\subsection*{ {\bf Proof of Theorem \ref{thm:rd_gpsqexp} } }

The proof follows from an application of Corollary \ref{thm:rd_gp} to the present setting. We assume $\sigma^2 = 1$ for this proof. 
Also at the very onset, we mention that we replace $\lambda_j$ by $a_n^{-1} e^{-j/a_n}$ subsequently, since after some algebra, it can be shown that $\lambda_j \asymp a_n^{-1} e^{-j/a_n}$. Recall $a_n = n^{1/(2 \alpha + 1)}$ and choose $k_n = n^{1/(2\alpha +1)} \log \big\{ n^{2\alpha/(2\alpha +1)} \big\}$ in Corollary \ref{thm:rd_gp}. We first verify that {\bf (C0)} -- {\bf (C3)} are satisfied. 

We start with {\bf (C0)}, which requires verifying {\bf (A1)} \& {\bf (A2)}. For {\bf (A1)}, we have $\norm{\Lambda^{-1}}_2 = \lambda_{k_n}^{-1} \asymp a_n e^{k_n/ a_n} \lesssim n$ by choice of $k_n$. From Lemma \ref{lem:eigbd}, we have that for any $j = 1, \ldots, k_n$, $\sup_{x \in \mb R} |\phi_j(x)| \le a_n^{1/4} e^{b k_n/a_n}$. Setting $L_n = a_n^{1/4} e^{b k_n/a_n}$, we have $L_n^2 k_n \log k_n \lesssim n^{(4 b \alpha + 3/2)/(2 \alpha + 1)} \log n \lesssim n$ as long as $\alpha > 1/\{4(1 - 2b)\}$, verifying {\bf (A2)}. 

We next verify {\bf (C1)}. Clearly, $k_n = o(n \epsilon_n^2)$. So it remains to establish that $\norm{\theta_0^t}_{\mb H}^2 = o(n \epsilon_n^2)$. Bound 
$$
\norm{\theta_0^t}_{\mathbb{H}}^2 \asymp a_n\sum_{j=1}^{k_n} e^{j/a_n} \theta_{0j}^2 \le a_n  \norm{\bft_0}_{\alpha}^2  \max_{1 \leq j \leq k_n} e^{j/a_n} j^{-2 \alpha}.
$$ 
The function $x \to e^{x/a} x^{-2 \alpha}$ is monotonically decreasing on the interval $(0, 2 \alpha a_n)$ and monotonically increasing on $[2\alpha a_n, \infty)$. Therefore, $\max_{1 \leq j \leq k_n} e^{j/a_n} j^{-2 \alpha} \le \max_{j \in \{1, k_n\} } e^{j/a_n} j^{-2 \alpha}$. We have $1/a_n < 1$, and hence $e^{j/a_n} j^{-2 \alpha}$ evaluated at $j = 1$ can be bounded above by $e$. 
$e^{j/a_n} j^{-2 \alpha}$ evaluated at $j = k_n$ is bounded above by $n^{2\alpha/(2\alpha +1)} k_n^{-2\alpha} = o(1)$. Hence $\norm{\theta_0^t}_{\mathbb{H}}^2 /n\epsilon_n^2 \lesssim a_n/n \epsilon_n^2 \to 0$ as $n \to \infty$. 

To verify {\bf (C2)}, we 
need to show that $\norm{f_0 - f_0^t}_{2, \rho} \lesssim \epsilon_n$. Indeed, $\norm{f_0 - f_0^t}_{2, \rho}^2 = \sum_{j=1}^{k_n +1} \theta_{0j}^2  \leq k_n^{-2\alpha} \sum_{j=k_n+1}^{\infty} j^{2\alpha} \theta_{0j}^2  \leq k_n^{-2\alpha} \norm{\bft_0}_{\alpha}^2 = o(\epsilon_n^2)$. 

It now remains to verify {\bf (C3)}. As noted in the paragraph after \eqref{eq:t2n_bd}, the numerator in \eqref{eq:prior_rat1} can be expressed as $\Pi(\norm{f - f^t}_{2, \rho}^2 > M^2\epsilon_n^2) = \Pi(\sum_{j=k_n+1}^{\infty} \lambda_j Z_j^2 > M^2 \epsilon_n^2)$ with $Z_j$s i.i.d. $\Gauss(0, 1)$. Noting that 
$\sum_{j=k_n+1}^{\infty}\lambda_j  \asymp e^{-k_n/a_n} =  n^{-2\alpha/ (2\alpha +1)} \le \epsilon_n^2 $, 
\begin{eqnarray}\label{eq:centerBerns}
\Pi \bigg(\sum_{j=k_n +1}^{\infty} \lambda_j Z_j^2 > M^2\epsilon_n^2\bigg) \leq   \Pi \bigg\{\sum_{j=k_n +1}^{\infty} \lambda_j (Z_j^2-1) > M^2\epsilon_n^2/2\bigg\}. 
\end{eqnarray}
 $(Z_j^2 - 1)$s are mean-zero sub-exponential random variables. By an application of Bernstein's inequality for linear combinations of mean-zero sub-exponential random variables (Proposition 5.16 of \cite{vershynin2010introduction}), 
\begin{align}
& \Pi \bigg\{\sum_{j=k_n +1}^{\infty} \lambda_j (Z_j^2-1) > M^2\epsilon_n^2/2\bigg\} \le 2 \exp \bigg [ -C' \min \bigg\{ \frac{M^4 \epsilon_n^4}{K^2 \sum_{j=k_n+1}^{\infty} \lambda_j^2}, \frac{M^2\epsilon_n^2}{K (\max_{j > k_n} \lambda_j)} \bigg \} \bigg ]  \nonumber \\
& \le 2 \exp \bigg [ -C \min \bigg\{ a_nM^4 \epsilon_n^4 e^{2k_n / a_n}, a_nM^2 \epsilon_n^2 e^{(k_n+1)/a_n} \bigg\} \bigg]  \nonumber \\
&= 2\exp \big [ -C \min \big\{ a_nM^4\log^4 n, a_nM^2 \log^2 n\big\} \big]  
= 2 \exp ( - C M^2 n^{1/(2\alpha +1)} \log^2 n),\label{eq:prior_tail}
\end{align}
where $K,C, C'$ are global constants. The second inequality in the previous display is due to $\sum_{j=k_n+1}^{\infty} \lambda_j^2 \asymp (1/2a_n) e^{-2k_n /a_n}$ and $\max_{j > k_n} \lambda_j = (1/a_n) e^{-(k_n+1)/a_n}$. 

Next, the term in the denominator of \eqref{eq:prior_rat1}, $\Pi(\norm{f - f_0}_{2, \rho} \leq \tilde{\epsilon}_n) =  \wt{\m W}(\norm{\bft - \bft_0}_{\ell_2} \le \tilde{\epsilon}_n)$, where $\bft = (\theta_1, \theta_2, \ldots)$ with $\theta_j \sim \Gauss(0, \lambda_j)$ and $\bft_0 = (\theta_{01}, \theta_{02}, \ldots)$. Set $\tilde{\epsilon}_n = C n^{-\alpha/(2 \alpha + 1)}$ for some constant $C$. We show below that 
\begin{align}\label{eq:small_ball}
\wt{\m W}(\norm{\bft - \bft_0}_{\ell_2} \le \tilde{\epsilon}_n) \ge \exp \{ -C' n^{1/(2\alpha +1)} \log^2 n\}.
\end{align}
We now establish \eqref{eq:small_ball}. Recall $\lambda_j \asymp a_n^{-1} e^{-j/a_n}$. 
Let $\theta^t = (\theta_1, \ldots, \theta_{k_n})^{\T}$ and recall $\theta_0^t$ defined similarly. 
Then, $\wt{\m W}\big(\norm{\bft - \bft_0}_{\ell_2} < \tilde{\epsilon}_n\big) \ge \wt{\m W}\big(\norm{\theta^t - \theta_0^t}^2 \le \tilde{\epsilon}_n^2/2\big) \, \wt{\m W}\big(\sum_{j=k_n +1}^{\infty} (\theta_j - \theta_{0j})^2 \le \tilde{\epsilon}_n^2/2\big)$. Using $\sum_{j=k_n + 1}^{\infty} \theta_{0j}^2 \le \norm{\theta_0}_{\alpha}^2 k_n^{-2 \alpha} = o(\tilde{\epsilon}_n^2)$, the second term can be bounded below by $\wt{\m W}\big(\sum_{j=k_n+1}^{\infty} \theta_j^2 \le \tilde{\epsilon}_n^2/4 \big)$. By Markov's inequality, 
\begin{align}\label{eq:sb_1}
\wt{\m W}\big(\sum_{j=k_n+1}^{\infty} \theta_j^2 \le \tilde{\epsilon}_n^2/4 \big) \ge 1 - 4/\tilde{\epsilon}_n^2 \sum_{j=k_n + 1}^{\infty} E \theta_j^2 & \asymp 1 - \frac{4}{a_n \tilde{\epsilon}_n^2} \sum_{j=k_n+1}^{\infty} e^{-j/a_n} \nonumber \\
& \ge 1 - \frac{4 e^{-k_n/a_n}}{ \tilde{\epsilon}_n^2} \ge 1/2
\end{align}
for large $C$. We used above that $\sum_{j=k_n+1}^{\infty} e^{-j/a_n} \le \int_{k_n}^{\infty} e^{-x/a_n} dx = a_n e^{-k_n/a_n}$ and $e^{-k_n/a_n} = n^{-2\alpha/(2\alpha + 1)}$. Therefore, it is enough to show the bound \eqref{eq:small_ball} for $\wt{\m W}\big(\norm{\theta^t - \theta_0^t}^2 \le \tilde{\epsilon}_n^2/2\big)$. By Anderson's inequality (Lemma \ref{lem:anderson} in Appendix),
\begin{align}\label{eq:apply_and}
\wt{\m W}\big(\norm{\theta^t - \theta_0^t}^2 \le \tilde{\epsilon}_n^2/2\big) \ge e^{-\frac{1}{2} \norm{\theta_0^t}_{\mb H}^2 } \, \wt{\m W}\big(\norm{\theta_t}^2 \le \tilde{\epsilon}_n^2/2\big). 
\end{align}
We have already shown that $\norm{\theta_0^t}_{\mb H}^2 \lesssim a_n$, so that $e^{-\frac{1}{2} \norm{\theta_0^t}_{\mb H}^2 } \ge e^{-C' n^{1/(2 \alpha + 1)}}$. Therefore, suffices to bound $\wt{\m W}\big(\norm{\theta_t}^2 \le \tilde{\epsilon}_n^2/2\big)$. Recall $\theta_j^2/\lambda_j \sim \chi_1^2$, therefore $\theta_j^2$ has a density $(\sqrt{2 \pi x})^{-1} a_n e^{j/(2a_n)} \exp(-a_n e^{j/a_n} x/2) \ind_{(0, \infty)}(x)$. 
Let $d \bfx$ denote $dx_1\ldots dx_{k_n}$ in short and set $D_n = a_n/\sqrt{2 \pi}$. Then,
{\small \begin{align}
\wt{\m W}\bigg(\sum_{i=1}^{k_n} \theta_i^2 \leq \tilde{\epsilon}_n^2/2\bigg) 
&= D_n^{k_n} e^{\frac{k_n(k_n+1)}{4a_n}} \int_{\sum_{j=1}^{k_n} x_j \le \tilde{\epsilon}_n^2/2}  \prod_{j=1}^{k_n} \frac{ \exp\big(-a_n e^{j/a_n} x_j/2 \big)}{\sqrt{x_j}} \ d\bfx \notag  \\
& \ge \bigg(\frac{D_n \tilde{\epsilon}_n}{\sqrt{2}} \bigg)^{k_n} e^{\frac{k_n(k_n+1)}{4a_n}} \int_{\sum_{j=1}^{k_n} x_j \le 1} \exp\bigg(  - \frac{a_n e^{k_n/a_n}}{4} \sum_{j=1}^{k_n} x_j \bigg) \prod_{j=1}^{k_n} x_j^{-1/2} d\bfx \notag  \\
& = \bigg(\frac{D_n \tilde{\epsilon}_n}{\sqrt{2}} \bigg)^{k_n} e^{\frac{k_n(k_n+1)}{4a_n}}  \frac{\Gamma(1/2)^{k_n}}{\Gamma(k_n/2)} 
\int_{t = 0}^{1} \exp\bigg(  - \frac{a_n e^{k_n/a_n}}{4} t \bigg) t^{k_n/2-1} dt. \label{eq:lb_centered1}
\end{align}}
From the first to the second line, we replace $j$ by $k_n$, perform a change of variable and drop the $\tilde{\epsilon}_n$ term appearing inside the exponent as $\tilde{\epsilon}_n < 1$. The last equality follows from 
the Dirichlet integral formula (Lemma \ref{lem:dir_formula} in Appendix). Using $\Gamma(1/2) = \sqrt{\pi}$ and the standard inequality (see, for example, \cite{abramowitz1965handbook}) $\Gamma(\alpha) \le \sqrt{2 \pi e^2} e^{-\alpha} \alpha^{\alpha - 1/2}$ for $\alpha > 1$, we can simplify \eqref{eq:lb_centered1} to write
{\small \begin{align}
\wt{\m W}\bigg(\sum_{i=1}^{k_n} \theta_i^2 \leq \tilde{\epsilon}_n^2/2\bigg)  \gtrsim \sqrt{k_n} e^{\frac{k_n(k_n+1)}{4a_n}}  \bigg(\frac{a_n \tilde{\epsilon}_n^2 e}{2k_n} \bigg)^{k_n/2} \int_{t = 0}^{1} \exp\bigg(  - \frac{a_n e^{k_n/a_n}}{4} t \bigg) t^{k_n/2-1} dt. \label{eq:lb_centered2}
\end{align}}
The integral in the above display can be bounded from below by $(2e^{-1}/k_n) (a_n e^{k_n/a_n}/4)^{-k_n/2}$. Substituting this bound and simplifying, the lower bound is  
$$
\frac{1}{k_n} e^{\frac{k_n(k_n+1)}{4a_n}}  e^{-\frac{k_n^2}{2 a_n}} \bigg(\frac{2 e\tilde{\epsilon}_n^2}{k_n} \bigg)^{k_n/2} \gtrsim e^{-C' n^{1/(2 \alpha + 1)} \log ^2 n}. 
$$
Combining with \eqref{eq:apply_and}, \eqref{eq:small_ball} is proved. 

Finally, the ratio of the bounds in \eqref{eq:prior_tail} and \eqref{eq:small_ball} converge to zero by choosing $M$ large enough, completing the proof.

\section*{Appendix}
\appendix

\section{Proof of Lemma \ref{lem:eigbd}}

It suffices to show that for any $t > 0$, 
$\max_{0 \le j \le a t} \sup_{x \in \mb R}|\phi_{j+1}(x)| \lesssim  a^{1/4} e^{b t}$ for large $a$. 
For fixed $b$, clearly $c = \sqrt{b^2 + 2 b a^2} \asymp a$. Therefore, $\phi_0(x) =  (c/b)^{1/4} \lesssim a^{1/4}$, so enough to take the $\max$ over $1 \le j \le at$. 
FInally, since both $\phi_{j+1}$ and $H_j$ are symmetric functions, it suffices to consider the supremum over $x$ on $(0, \infty)$.  

The Hermite polynomials have an integral representation
\begin{align}\label{eq:integral_rep}
H_j(z) = \frac{2^j}{\sqrt{\pi}} \int_{-\infty}^{\infty} (z + i t)^j e^{-t^2} dt. 
\end{align}
For $z > 0$, using $| \int f d \mu | \le \int |f| d \mu$, we have 
\begin{align}\label{eq:integral_bd}
| H_j(z) | 
& \leq \frac{2^j}{\sqrt{\pi}} \, \int (z^2 + t^2)^{j/2} e^{-t^2} dt \nonumber \\
& =  \frac{2^{j+1} z^{j+1}}{\sqrt{\pi}} \int_{t = 0}^{\infty} (1 + t^2)^{j/2} e^{-z^2 t^2} dt. 
\end{align}
Let $g(t) = (1 + t^2)^{j/2} \, e^{- z^2 t^2/2}$; 
clearly, $\log g(t) = (j/2) \log(1 + t^2) - z^2 t^2/2$. Differentiating, $\frac{d}{dt} \log g(t) = jt/(1 + t^2) - z^2 t$. Setting $\frac{d}{dt} \log g(t) = 0$, we have $t \{(1 + t^2) - j/z^2\} = 0$. Therefore, if $z^2 > j$, $g(t)$ attains maxima at $t = 0$ with $g(0) = 1$. On the other hand, if $z^2 \leq j$, $g(t)$ attains maxima at $t = (j/z^2 - 1)^{1/2}$.  
When $z^2 > j$, bounding $g(t) \leq 1$ in \eqref{eq:integral_bd}, we get the inequality
\begin{align}\label{eq:hermite_bd2}
|H_j(z)|  & = \frac{2^{j+1} z^{j+1}}{\sqrt{\pi}} \int_{t = 0}^{\infty} g(t) e^{-z^2 t^2/2} dt \nonumber \\  
             & \leq \frac{2^{j+1} z^{j+1}}{\sqrt{\pi}} \int_{t = 0}^{\infty} e^{-z^2 t^2/2} dt \nonumber \\
             & = \frac{2^{j+1} z^{j+1}}{\sqrt{\pi}} \frac{1}{2} \, \frac{\sqrt{2 \pi}}{z} \nonumber \\
             & = \sqrt{2} \, 2^j z^j. 
\end{align}
We record this bound below:
\begin{lemma}\label{lem:hermite_altbd}
The Hermite polynomials satisfy $|H_j(z)| \le \sqrt{2} \, 2^j z^j$ whenever $z^2 > j$. 
\end{lemma}
Note that the exponential $e^{z^2/2}$ term in Cramer's bound has been replaced by a polynomial $z^j$ term. When $z^2 = j$, ignoring constants, Cramer' bound for $|H_j(z)|$ is $2^{j/2} \sqrt{j!} \, e^{j/2}$ while the same from Lemma \ref{lem:hermite_altbd} is $2^{j} j^{j/2}$. Using $j! \asymp j^{j+1/2} e^{-j}$, we see that both bounds give similar results when $z^2 \asymp j$. 

As discussed at the beginning, we now proceed to establish the bound for $|\phi_{j+1}(x)|$ for $1 \le j \le at$ and $x > 0$. For $x \in (0, \sqrt{t})$, use Cramer's bound to obtain $|\phi_{j+1}(x)| \lesssim (c/b)^{1/4} \, e^{bx^2} \lesssim a^{1/4}e^{bt}$ when $x \in (0, \sqrt{t})$. 

When $x > \sqrt{t}$, setting $z = \sqrt{2c} x$, we have $z^2 > 2ct > j$ for any $j \le a t$. Therefore, we have two bounds for $|H_j(z)|$: (i) $|H_j(z)| \lesssim \sqrt{2^j j!} \, e^{z^2}$ from Cramer's bound, and (ii) $|H_j(z)| \lesssim 2^j z^j$ from Lemma \ref{lem:hermite_altbd}. Using a combination of both delivers a tighter bound for $|\phi_{j+1}(x)|$. Let $\delta > 0$ be such that $c \delta > b$. Then, for  any such $\delta$, we may write
\begin{align}\label{eq:hermite_comb_bd}
| H_j(z)| 
& = |H_j(z) |^{1 - \delta} |H_j(z) |^{\delta} \nonumber \\
& \lesssim \big\{ 2^{j(1 - \delta)/2} (j !)^{(1 - \delta)/2} e^{c(1 - \delta) z^2} \big\} \, \big\{2^{j \delta } (2c)^{j \delta/2} z^{j \delta} \big\} \nonumber \\
& = 2^{j/2 + j \delta} (j!)^{(1 - \delta)/2} c^{j \delta/2} \, z^{j \delta} e^{c(1 - \delta) z^2}.
\end{align}
Substituting this bound in the expression for $\phi_{j+1}$, we have
\begin{align*}
| \phi_{j+1}(x) | 
& \lesssim (c/b)^{1/4} \, \frac{2^{j \delta} c^{j \delta/2} }{ (j !)^{\delta/2} } \, x^{j \delta} e^{- (c \delta - b) x^2}. 
\end{align*}
The function $x \to x^{j \delta} e^{-(c \delta - b) x^2}$ for $x > 0$ achieves its maximum at $x = [j \delta/\{2 (c \delta - b) \}]^{1/2}$. Substituting $x^2 = j \delta/\{2(c \delta - b)\}$ in the above display and bounding $j ! \geq (j/e)^j$, 
\begin{align*}
| \phi_{j+1}(x) | 
& \lesssim c^{1/4} \, \frac{2^{j \delta} c^{j \delta/2} }{ (j/e)^{\delta/2} } \bigg\{ \frac{j \delta}{2 (c \delta - b)} \bigg\}^{j \delta/2} e^{- j \delta/2} \\
& = c^{1/4} \, 2^{j \delta/2} \bigg( \frac{c \delta}{c \delta - b} \bigg)^{j \delta/2}.
\end{align*}
Now choose $\delta = be/\{c(e-2)\}$, so that $c\delta/(c \delta - b) = e/2$. Then we have $| \phi_{j+1}(x) | \lesssim c^{1/4} \, e^{j \delta/2} \lesssim a^{1/4} e^{bt}$, since $j \delta/2 < a t\delta/2 \lesssim c t \delta/2 \lesssim b t$.

\section{Some useful results} \label{app}

Some matrix inequalities. Proofs can be found in standard texts; see for example, \cite{bhatia1997matrix}. 
\begin{lemma}\label{lem:tprod_frob}
For any two matrices $A, B$, 
\begin{align}
& s_{\min}(A) \norm{B}_F \leq \norm{AB}_F \leq \norm{A}_2 \norm{B}_F \tag{i} \\
& s_{\min}(A) \norm{B}_2 \leq \norm{AB}_2 \leq \norm{A}_2 \norm{B}_2. \tag{ii}
\end{align}
If $s_{\min}(A) \ge \norm{B}_2$, then
\begin{align}
& s_{\min}(A-B) \ge s_{\min}(A) - \norm{B}_2. \tag{iii}
\end{align}
\end{lemma}

A version of Anderson's lemma from \cite{van2008reproducing} which provides a sharp bound on the probability of shifted balls under multivariate Gaussian distributions in terms of the centered probability and the size of the shift. 
\begin{lemma}\label{lem:anderson}
Suppose $\bfxi \sim \mbox{N}_n(0, \Sigma)$ with $\Sigma$ p.d. and $\bfxi_0 \in \mathbb{R}^n$. Let $\norm{\bfxi_0}_{\mathbb{H}}^2 = \bfxi_0^{\T} \Sigma^{-1} \bfxi_0$. Then, for any $t > 0$,
\begin{align*}
P(\norm{\bfxi - \bfxi_0}_2 < t) \ge e^{- \frac{1}{2} \; \norm{\bfxi_0}_{\mathbb{H}}^2}  P(\norm{\bfxi}_2 \leq t/2). 
\end{align*}
\end{lemma}

\medskip
The Dirichlet integral formula (formula 4.635 in \cite{gradshteyn1980corrected}) to simplify integrals over the unit probability simplex.
\begin{lemma}\label{lem:dir_formula}
Let $\psi(\cdot)$ be a Lebesgue integrable function and $\alpha_j > 0, j = 1, \ldots, n$. Then,
\begin{eqnarray*}
\int_{\sum x_j \leq 1} \psi\big(\sum x_j\big) \prod_{j=1}^n x_j^{\alpha_j - 1} d \bfx = \frac{\prod_{j=1}^n \Gamma(\alpha_j)}{\Gamma \big(\sum_{j=1}^n \alpha_j\big)} \int_{t=0}^1 \psi(t) \, t^{( \sum \alpha_j) - 1} dt. 
\end{eqnarray*}
\end{lemma}

\bibliographystyle{abbrv}
\bibliography{mybibfile}
\end{document}